\documentclass[12pt,thmsa, reqno]{amsart}

\usepackage{graphicx}
\usepackage{amssymb, euscript}
\usepackage[square, comma, sort&compress, numbers]{natbib}
\usepackage{amsmath}
\usepackage{multicol}

\def\snm1{\bbs^{n-1}}
\def\sn{\bbs^{n}}
\def\snm1{\bbs^{n-1}}
\def\rnp1{\bbr^{n+1}}

\def\Cal{\mathcal}

\def\C{{\Cal C}}

\def\R{{\Cal R}}
\def\M{{\Cal M}}

\def\S{{\Cal S}}
\def\U{{\Cal U}}
\def\V{{\Cal V}}
\def\F{{\Cal F}}

\def\N{{\Cal N}}

\def\bbr{{\Bbb R}}

\def\bbs{{\Bbb S}}

\def\const{{\hbox{\rm const}}}

\def\rn{\bbr^n}

\def\part{\partial}
\def\intl{\int\limits}

\def\Gam{\Gamma}
\def\Om{\Omega}
\def\a{\alpha}
\def\om{\omega}

\def\del{\delta}
\def\vp{\varphi}

\def\gam{\gamma}

\def\sig{\sigma}
\def\lam{\lambda}

\def\th{\theta}
\def\e{\varepsilon}
\def\t{\tau}

\def\chi{{\bf 1}}

\font\frak=eufm10

\def\fr#1{\hbox{\frak #1}}

\def\frC{\fr{C}}

\newtheorem{theorem}{Theorem}[section]
\newtheorem{lemma}[theorem]{Lemma}

\theoremstyle{definition}

\newtheorem{example}[theorem]{Example}

\theoremstyle{remark}
\newtheorem{remark}[theorem]{Remark}

\theoremstyle{corollary}
\newtheorem{corollary}[theorem]{Corollary}
\newtheorem{proposition}[theorem]{Proposition}

\numberwithin{equation}{section}

\newcommand{\be}{\begin{equation}}
\newcommand{\ee}{\end{equation}}

\newcommand{\bea}{\begin{eqnarray}}
\newcommand{\eea}{\end{eqnarray}}
\newcommand{\Bea}{\begin{eqnarray*}}
\newcommand{\Eea}{\end{eqnarray*}}

\def\sideremark#1{\ifvmode\leavevmode\fi\vadjust{\vbox to0pt{\vss
 \hbox to 0pt{\hskip\hsize\hskip1em
\vbox{\hsize2cm\tiny\raggedright\pretolerance10000
 \noindent #1\hfill}\hss}\vbox to8pt{\vfil}\vss}}}%

\begin{document}

\title[Reconstruction of functions on the sphere]
{Reconstruction of functions on the sphere from their integrals over
Hyperplane Sections}


\author{B. Rubin }
\address{Department of Mathematics, Louisiana State University, Baton Rouge,
Louisiana 70803, USA}
\email{borisr@lsu.edu}





\begin{abstract}  We obtain  new inversion formulas for the   Funk type transforms of two kinds associated to spherical sections by hyperplanes passing through a common  point $A$ which lies inside the $n$-dimensional unit sphere or on the sphere itself. Transforms of the first kind are defined by integration over complete subspheres and can be  reduced to the classical Funk transform.  Transforms of the second kind perform integration over truncated subspheres, like spherical caps or bowls, and can be reduced to the hyperplane Radon transform.
 The main tools are analytic families of $\lam$-cosine transforms, Semyanisyi's integrals,   and modified stereorgraphic projection with the pole at  $A$. Assumptions for functions are close to minimal.

\end{abstract}

\maketitle

\section{Introduction}
\setcounter{equation}{0}

Spherical integral geometry covers a wide range of problems related to reconstructing  functions on the Euclidean sphere  from their integrals  over prescribed submanifolds. Problems of this kind arise in  tomography, convex geometry, and many other areas; see, e.g., \cite{AKK, GGG2, GRS, H11, HMS, Kaz, Ku14, P1, P2, Ru15, ZS}.
A typical example is the Funk transform \cite{Fu11, Fu13} that grew up from the work of Minkowski \cite{Min}.
In the $n$-dimensional setting, this transform integrates a function on the unit sphere $\sn$ in $\bbr^{n+1}$  over cross-sections of this sphere by hyperplanes passing through the origin $o$ of $\bbr^{n+1}$; see (\ref{fu}) below.

In the present paper we obtain  new results related to the  similar  transforms  over spherical sections by hyperplanes passing through an arbitrary fixed point $A$ which lies either inside the sphere or on the sphere itself. 
A distinctive feature of such transforms in comparison with the Funk case is that if $A\neq o$, they do not commute with {\it all} orthogonal transformations and need special tools for their investigation. Modifications of these transforms act on functions which are defined  on 
truncated  spheres, like spherical caps or bowls. An important problem is to  reconstruct these functions  from their integrals over the corresponding planar slices.

\vskip 0.2truecm

\noindent {\bf Setting of the Problem, Brief History, and Main Results.}

\vskip 0.2truecm

Given a  point $A$ in $\rnp1$, 
let $\t_A$ be a hyperplane  passing through $A$. We assume $A=(0, \ldots, 0,a)$,  $-1<a\le 1$, and set \[\sn_a=\{\eta=(\eta_1, \ldots, \eta_{n+1})\in \sn:\, \eta_{n+1}<a\}.\] Our main concern is the Radon-like transforms
\be\label{843ft}
(\F_a f)(\t_A)=\intl_{\t_A \cap \,\sn} f(\eta) \, d_{\t_A} \eta, \quad (\S_a f)(\t_A)=\intl_{\t_A \cap \,\sn_a} f(\eta) \, d_{\t_A} \eta,\ee
associated with  spherical slices by the hyperplanes $\t_A$.
These transforms coincide when $a=1$, but have different features if $a<1$. The case $a=0$ for $\F_a$ gives the classical Funk transform. In the case, when  $a=1$, $n=2$, and $f$ is smooth, these operators were introduced by Abouelaz and  Daher \cite{AbD} and studied by Helgason \cite [p. 145] {H11}. The case $a=1$ for all $n\ge 2$ was studied in \cite[Section 7.2]{Ru15}, \cite[Section 5]{Ru17} in the general context of Lebesgue integrable functions.

Operators $\F_a$ for $0\le a<1$ and $n=2$ were studied by Quellmalz \cite{Q}, who suggested a nice change of variable in the meridional direction and reduced the inversion problem for $\F_a f$ to the similar one for the Funk transform. The function $f$ was assumed either continuous or belonging to $L^2(\bbs^2)$.\footnote{After the manuscript had been completed, the author became aware of the recent preprint by Quellmalz \cite{Q1}, where the results of \cite{Q} were extended to all $n\ge 2$. The method of \cite{Q1} essentially differs from ours.}
The injectivity of $\F_a$ with $a=1$ was  obtained in \cite{Q} as a limiting case under certain smoothness assumption for $f$.
Another point of view on the inversion problem for $\F_a$  and more general operators, related to planar  sections of $\bbs^2$, was suggested by Gindikin et al.   \cite{GRS} in terms of the  kappa-operator.

The integrals $\S_a f$ were studied by Salman \cite{Sa16, Sa17} when $0\le a<1$, $n\ge 2$, and $f$ is a continuously  differentiable function with compact support strictly inside $\sn_a$. He used the  stereorgraphic projection to reduce inversion of $\S_a f$ to the similar problem for a certain Radon-like transform over spheres in $\rn$.  The latter was explicitly inverted by making use some technique borrowed from John \cite{Jo4}. The case $a=1$ was covered in \cite{Sa16, Sa17} by passing to the limit as $a\to 1$.

 The aim of the present paper is twofold. Firstly, we extend Quellmalz's approach from \cite{Q} to all $n\ge 2$ and all $a\in (-1,1)$, and perform this extension under minimal assumptions for $f$. Specifically, we assume $f\in L^1(\sn)$, as in the Funk case. To circumvent the main technical difficulty related to the change of variables, we consider $\F_a$ as a member of the analytic family of the relevant $\lam$-cosine transforms for which the corresponding Jacobians can be easily computed.

Secondly, we suggest a new approach to the study of the operators $\S_a$. Instead  of the stereographic projection with the pole $(0, \ldots, 0,1)$, as in  \cite{Sa17}, we choose the pole at $(0, \ldots, 0,a)$. The case $a=1$ agrees with the usual stereographic projection, as in \cite{H11} for $n=2$. As a result, $\S_a$ expresses through the classical hyperplane Radon transform $Rf$ of functions on $\rn$ and all known properties of the latter, including inversion formulas, transfer to  $\S_a$. To perform this transition, we invoke  analytic families  of the  $\lam$-cosine transforms and Semyanistyi's fractional integrals \cite{Se1}, \cite[Section 4.9]{Ru15}. These remarkable integrals  include the Radon transform as a particular case. As an interesting by-product, we obtain a  formula
\be\label{lttaw}
\frC_a^\lam = \V_{a, \lam} R^\lam \U_{a, \lam},\ee
establishing connection between the $\lam$-cosine transform $\frC_a^\lam$ and the Semyanistyi transform $R^\lam$ via certain bijections  $\V_{a, \lam}$ and $\U_{a, \lam}$ (see Lemma \ref {kiaw}). Similar formula was known before (see, e.g., \cite[Section 5.2]{Ru15}) only for $\lam =-1$ and $a=0$, that corresponds to the Funk  and Radon transforms. The limiting case $\lam =-1$
in (\ref{lttaw}) gives the desired connection between $\S_a$ and $R$.

It is worth noting that  the $\lam$-cosine transforms, Semyanistyi's integrals, and their modifications have proved to be   important objects in integral geometry and  harmonic analysis (see, e.g., \cite {OPR, OR, Ru98d,  Ru02b, Ru02c, Ru08, Ru13a, Ru13c, Ru15}) because they enable one to obtain diverse results for particular values of $\lam$ by making use of analytic continuation. We believe that  (\ref{lttaw}) paves the way to new investigations in the area.

For possible applications and further purposes, we separately mention particular cases $a=0$, $a=1$, and the case of zonal functions, when all formulas look much simpler.

The paper is organized as follows.
Section 2 contains necessary preliminaries.  Section 3 and 4 are devoted to the operators $\F_a$ and $\S_a$, respectively. The Appendix contains technical results related to limiting properties of some integrals depending on a parameter.

\section{Preliminaries}

\subsection {Notation}  In the following $\sn$, $n\ge 2$, denotes the unit sphere in the real Euclidean space $\rnp1$; $e_1, \ldots,  e_{n+1}$ are the coordinate unit vectors;  the space $\bbr^n= \bbr e_1 \oplus \cdots \oplus \bbr e_{n}$ is considered as a coordinate hyperplane of $\rnp1$.
For $\eta =(\eta_1, \ldots, \eta_{n+1})\in  \sn$, $d\eta$ stands for the  Riemannian  measure on $\sn$;
$\sigma_{n} \equiv \int_{\sn} d\eta= 2\pi^{(n+1)/2} \big/ \Gamma ((n+1)/2)$.  We write $d_*\eta= d\eta/\sigma_{n}$ for  the normalized surface element on $\sn$; $\eta^\perp$ is the linear subspace orthogonal to $\eta$. The notation
\[ p_N=(0, \ldots, 0,1), \qquad p_S=(0, \ldots, 0,-1)\]
will be used for the North Pole and the South Pole of $\sn$, respectively. We set
$\sn_+=\{ \eta \in \sn :\, \eta_{n+1} \ge 0\}$,   $\;\sn_-=\{ \eta \in \sn :\, \eta_{n+1} <0\}$,
\[ A=(0, \ldots, 0,a), \qquad -1< a\le 1.\]

The notation for the function spaces $L^p(\sn)$,  $C(\sn)$,   $C^\infty(\sn)$ and the similar spaces on $\rn$ is standard. We write $C_c(\rn)$ and    $C_c^\infty(\rn)$ for the corresponding spaces of compactly supported functions.

 Dealing with  integrals, we say that  the integral
 exists in the Lebesgue sense if it is finite when the expression under the sign of integration is replaced by its absolute value.
 The letter $c$ (sometimes with subscripts)  stands  for an inessential positive constant that may vary at each  occurrence.

\subsection {Riesz potentials}

The  Riesz potential of a  function $g$ on $\bbr$ is defined by the formula
\be\label{rpot} (I^\a g)(x)=\frac{1}{\gamma_1(\a)}
\intl_{\bbr}
 \frac{g(y)\,dy}{|x-y|^{1-\a}},\qquad
  \gamma_1(\a)=
  \frac{2^\a\pi^{1/2}\Gamma(\a/2)}{\Gamma((1-\a)/2)}.
  \ee
We will need this operator only with $0<\a<1$, though other values of $\a$ can be considered.
The following  fact is a particular case of the more general statement in \cite[Lemma 3.2]{Ru15}.

\begin{lemma}\label {riesz-enLEM00}  Let  $g$ be a continuous function in some neighborhood $\Om$ of $x\in \bbr$, and let
\be\label {riesz-en} \intl_{\rn \setminus \Om} |y\,g(y)|\, dy <\infty.\ee
Then
\be\label{apprpr-ty} \lim\limits_{\a\to 0} \,(I^\a g)(x)=g (x).\ee
\end{lemma}

A generalization of this statement to the case when the function under the sign of integration  depends also on $\a$ (see Lemma \ref{riesz-enLEM} in Appendix) will be used in the present paper.

\subsection {The Funk Transform and the Spherical Slice Transform}

The {\it Funk transform} of a continuous function $f$ on $\sn$ has the form
\be\label{fu} (\F f)(\xi)=\frac{1}{\sig_{n-1}}\intl_{\sn \cap \xi^\perp} f(\eta)\, d_\xi\eta, \qquad \xi \in \sn,\ee
where $d_\xi\eta$ stands for the Riemannian surface measure on $\sn \cap \xi^\perp$.  If $f\in L^1(\sn)$, then $(\F f)(\xi)$ is finite for almost all $\xi\in \sn$ and
the operator $\F$ extends as a linear bounded operator from $L^1(\sn)$ to $L^1(\sn)$. Properties of the Funk transform are described in numerous publications; see, e.g., \cite{GGG2, H11, P1, Ru15}.
In particular, if $\eta =(\eta_1, \ldots, \eta_{n+1})\in \sn$ and $f(\eta)\equiv f_0 (\eta_{n+1})$, i.e.
$f$ is a zonal function,  then $\F f$ is zonal too. If, for $f$ zonal, we set $(\F f)(\xi)=F_0(\rho)$, where $\rho=\sqrt {1-\xi_{n+1}^2}$ is  the sine of the geodesic distance between $\xi$ and the North Pole $p_N$,
then
\be\label{ierertw}
\F_0 (\rho)=c\,\rho^{2-n}\intl_0^\rho f_0 (t)(\rho^2 -t^2)^{(n-3)/2}\, dt, \qquad c=\frac{2\sig_{n-2}} {\sig_{n-1}};\ee
cf. \cite [Theorem 5.21]{Ru15}. Abel type integrals of the form (\ref{ierertw}) were studied in  many sources; see, e.g., \cite {Ru15, SKM}.

We shall consider  modifications of (\ref{fu}) for the non-central cross-sections of $\sn$ by the hyperplanes passing through an arbitrary  fixed point inside the sphere $\sn$ or on the sphere itself. Without loss of generality, we  choose such a point to be  $A=(0, \ldots, 0,a)$, $-1< a\le 1$. Replacing $\xi^\perp$ in (\ref{fu}) by the shifted hyperplane $\xi^\perp +ae_{n+1}$, we obtain the collection of spherical slices
\bea \gam_a (\xi)&=& \{\eta \in \sn: \xi \cdot (\eta -a e_{n+1})=0\} \nonumber\\
\label{sl} &=& \{\eta \in \sn: \xi \cdot \eta =a \xi_{n+1}\},
\eea
labeled by the point $\xi$. The corresponding {\it spherical slice transform} is defined by
\be\label{fua} (\F_a f)(\xi)=\intl_{\gam_a (\xi)} f(\eta)\, d_{\gam_a}\eta, \qquad \xi \in \sn,\ee
where $d_{\gam_a}\eta$ stands for the  surface area measure on $\gam_a (\xi)$. Clearly, $\gam_a (-\xi) = \gam_a (\xi)$ and $(\F_a f)(-\xi)= (\F_a f)(\xi)$.
Every slice $\gam_a (\xi)$ is an $(n-1)$-dimensional Euclidean sphere with center at $(a\xi_{n+1})\xi$ and radius $(1-a^2 \xi_{n+1}^2)^{1/2}$. Hence
\be (\F_a 1)(\xi)=\sig_{n-1} (1-a^2 \xi_{n+1}^2)^{(n-1)/2}.\ee
This simple equality implies the following statement.
\begin{proposition} Let $-1< a\le 1$, $n\ge 2$. The slice transform $\F_a$ is a linear bounded operator on $L^\infty (\sn)$.
\end{proposition}

If $-1< a< 1$, i.e., $a=1$ is excluded, this preposition extends to $L^p(\sn)$ for all $p\in [1,\infty]$ and to $C(\sn)$; see  Corollary \ref{cor1} below.

The spherical slice transform (\ref{fua}) agrees with the classical  spherical mean
\bea \label{sphm0}
(M_{\xi} f)(t)\!\!&=&\!\!\frac{1}{\sig_{n-1}}\intl_{\sn\cap \xi^\perp}
\!\!f( \sqrt{1-t^2}\, \eta +t\xi)\,d_\xi\eta\\
\label{sphm01}\!\!&=&\!\!\frac{(1\!-\!t^2)^{(1-n)/2}}{\sig_{n-1}}\intl_{\xi \cdot \eta =t} \!f(\eta)\, d_{\xi, t} \eta, \quad -1<t<1,\eea
where $d_\xi\eta\, $  and  $d_{\xi, t} \eta$  stand for the  corresponding surface area measures.
In the limiting case $t\to \pm 1$ we have $(M_{\xi} f)(\pm \,1)= f(\pm \,\xi)$. Clearly,
\be\label{fuat} (\F_a f)(\xi)=\sig_{n-1} (1-a^2 \xi_{n+1}^2)^{(n-1)/2} (M_{\xi} f)(a\xi_{n+1}).\ee

Numerous properties of the spherical mean $(M_{\xi} f)(t)$ are described in \cite[pp. 503-508]{Ru15}. In particular,
\be\label {clop10} \intl_{\sn} h(\xi
\cdot \eta)\, f(\eta)\, d\eta=\sig_{n-1} \intl_{-1}^1 h(t)  (M_{\xi} f)(t)\,(1-t^2)^{(n-2)/2}\,dt\ee
provided that at least one of these integrals exists in the Lebesgue sense.
In the case $h\equiv 1$ (\ref{clop10})  has a simpler form
 \be \label{cnhos} \intl_{\sn} f(\eta)\, d\eta= \intl_{-1}^1  (1-t^2)^{(n-2)/2}\,
 dt\intl_{\snm1} \!\!f \left( \sqrt{1-t^2}\, \psi +te_{n+1}\right)\,d\psi;\ee
 see \cite[Lemma 1.34]{Ru15}.

\subsection {The $\lam$-Cosine Transform}
In many occurrences it is convenient to consider the Funk transform (\ref{fu}) as a member of the meromorphic family of {\it normalized cosine (or $\lam$-cosine) transforms}
\be \label{cos} (\C^\lam f)(\xi) =\gam_{n,\lam}
 \intl_{\sn} f(\eta)\, |\xi \cdot \eta|^{\lam} \,d_*\eta, \ee
where $\, d_*\eta=d\eta/\sig_n$,  $\;Re \,\lam >-1$, $ \;\lam \neq 0,2,4, \ldots$,
  \be\label{siefu} \gam_{n,\lam}\!=\!\frac{\pi^{1/2}\,\Gamma( -\lam/2)}{\Gamma ((n\!+\!1)/2)\, \Gamma ((1\!+\!\lam)/2)}.\ee
This transform has a long history; see \cite [pp. 283, 361, 527] {Ru15}, \cite{OPR} and references therein. We restrict to the case $-1<\lam <1$, which is sufficient for our purposes.

\begin{lemma} \label{passlim} {\rm (cf. \cite [p. 531]{Ru15})} If $f$ is a continuous function on $\sn$, then for  every $ \xi \in\sn$,
\be\label{lim1}
\lim\limits_{\lam \to -1} (\C^\lam f)(\xi)= c_n\, (\F f)(\xi),\qquad c_n= \frac{\pi^{1/2}}{\Gamma (n/2)}\,.
\ee
\end{lemma}

Both transforms (\ref{fu}) and  (\ref{cos}) annihilate odd functions that form the kernel (the null space) of the corresponding operators.

 We introduce the  ``shifted" cosine transform
\be \label{cosa} (\C_a^\lam f)(\xi) =\gam_{n,\lam}
 \intl_{\sn} f(\eta)\, |\xi \cdot (\eta -ae_{n+1})|^{\lam} \,d\eta, \quad -1<a\le 1,\ee
where $\lam$ and $\gam_{n,\lam}$  are the same as in (\ref{cos}).

\begin{lemma} \label{passlim2} If $f$ is a continuous function on $\sn$, then for  every $ \xi \in\sn$  and $|a\xi_{n+1}|<1$,
\be\label{lim1}
\lim\limits_{\lam \to -1} (\C_a^\lam f)(\xi)\!= \!\frac{d_n}{\sqrt {1\!-\!a^2\xi_{n+1}^2}}  \, (\F_a f)(\xi),\quad d_n\!= \!\frac{\pi}{\Gamma ((n\!+\!1)/2)}\,.
\ee
\end{lemma}
\begin{proof} By (\ref{clop10}),
\[(\C_a^\lam f)(\xi) =\gam_{n,\lam} \sig_{n-1}\intl_{-1}^1 |t -a\xi_{n+1}|^{\lam}   (M_{\xi} f)(t)\,(1-t^2)^{(n-2)/2}\,dt.\]
For every fixed $\xi$,  this integral is a constant multiple of the  Riesz potential
\be\label{rpot} (I^\a g)(x)=\frac{1}{\gamma_1(\a)}
\intl_{-1}^1  \frac{g(t)\,dt}{|x-t|^{1-\a}},\qquad   \gamma_1(\a)=
  \frac{2^\a\pi^{1/2}\Gamma(\a/2)}{\Gamma((1-\a)/2)}
  \ee
  (cf. (\ref{rpot})) with
\[ \a=\lam +1, \qquad x= a\xi_{n+1}, \qquad g(t)=(M_{\xi} f)(t)\,(1-t^2)^{(n-2)/2}.\]
Hence, by Lemma  \ref{riesz-enLEM00}  and (\ref {fuat}), a simple calculation gives
   \bea \lim\limits_{\lam\to -1}(\C_a^\lam f)(\xi)&=& d_n\sig_{n-1} (M_{\xi} f)(a\xi_{n+1})\,(1-a^2\xi_{n+1}^2)^{(n-2)/2}\nonumber\\
&=&\frac{d_n}{\sqrt {1\!-\!a^2\xi_{n+1}^2}}  \, (\F_a f)(\xi).\nonumber\eea
\end{proof}

\section {Spherical Slice Transforms via the Funk Transform}

In this section we extend   Quellmalz's method for the slice transform $\F_a$ (see \cite{Q} for $n=2$) to all $n\ge 2$. Everywhere in this section we assume  $-1< a<1$,  that is, $a=1$ is excluded.
We obtain an inversion formula for $\F_a$ and clarify the structure of the kernel of this operator. The basic idea is to express points in $\sn$ in spherical coordinates and change variable in the meridional direction in a proper way.

Specifically, we present a generic point $\eta \in \sn$ in the form
\be\label{eta} \eta \equiv \eta (\psi, u)=\sqrt {1-u^2} \psi + u e_{n+1}, \qquad \psi \in \snm1, \quad |u| \le 1,\ee
and consider the bijective map
\be\label{map} \sn \ni \eta (\psi, u) \xrightarrow{\;\mu\;} \tilde\eta (\psi, v)\in \sn,  \ee
where $u$ and $v$ are related by the formula
\be\label{form}  u=\frac{v+a}{1+av}.\ee
The inverse map $\mu^{-1}$ acts by the rule
\be\label{rule}
\mu^{-1}: \tilde\eta (\psi, v) \to \eta (\psi, u)=\tilde\eta\left (\psi, \frac{u-a}{1-au}\right )\ee
 and moves the ``equator" $v=0$ to the ``parallel"   $u=a$
while keeping the poles $p_N$ and $p_S$ fixed.
For the sake of convenience, we  abuse the notation and write $f(\psi, u)$ in place of $f(\eta (\psi, u))$.

We will  need one more bijective map
\be\label{mapnu} \sn \ni \xi (\vp, s) \xrightarrow{\;\nu\;} \tilde\xi (\vp, t)\in \sn,  \ee
where
\be\label{formnu}  t=s\frac{\sqrt {1-a^2}}{\sqrt {1-a^2s^2}}.\ee
Then
\be\label{mapnu1}
\nu^{-1}:  \, \tilde\xi (\vp, t) \to \xi (\vp, s) = \tilde\xi \left (\vp, s\frac{\sqrt {1-a^2}}{\sqrt {1-a^2s^2}} \right).\ee

\begin{lemma} \label{lem31}  If $f$ and $\Phi$ belong to $\in L^1 (\sn)$, then
\be\label{fff}\intl_{\sn} f(\eta)\, d\eta= (1-a^2)^{n/2} \intl_{\sn} \frac{(f\circ \mu^{-1})(\tilde\eta)}
{(1+a\tilde\eta_{n+1})^n}\, d\tilde\eta,\ee
\be\label{ggg}\intl_{\sn} \Phi(\tilde\xi)\, d\tilde\xi= \sqrt {1-a^2} \intl_{\sn} \frac{(\Phi\circ \nu)(\xi)}
{(1-a^2\xi^2_{n+1})^{(n+1)/2}}\, d\xi.\ee
\end{lemma}
\begin{proof}  By (\ref {cnhos}),
 \[ \intl_{\sn} f(\eta)\, d\eta= \intl_{-1}^1  (1-u^2)^{(n-2)/2}\,
 du\intl_{\snm1} \!\!f( \sqrt{1-u^2}\, \psi +ue_{n+1})\,d\psi.\]
 The result follows if we change variables  according to (\ref{form}),  taking into account  that
\[ \frac{du}{dv}=\frac{1-a^2}{(1+av)^2}, \qquad 1-u^2= \frac{(1-a^2)(1-v^2)}{(1+av)^2}. \]
The proof of the second statement is similar.
\end{proof}

Our next aim is to convert the slice transform (\ref{fua}) into the Funk transform  (\ref{fu})  by making use of the map (\ref{map}). To circumvent technicalities related to Jacobians, we invoke  Lemma \ref{passlim2} for the shifted cosine transform.
Given $-1< a<1$ and $\lam \in \bbr$, let
\bea\label{ma} (\M_{a, \lam} f)(\tilde\eta)&=&\frac{(f\circ \mu^{-1})(\tilde\eta)}
{(1+a\tilde\eta_{n+1})^{n+\lam}}, \\
\label{na} (\N_{a, \lam} \Phi)(\xi)&=&(1-a^2)^{(n+\lam)/2} (1-a^2\xi_{n+1}^2)^{\lam/2} (\Phi\circ \nu)(\xi), \quad \eea
where  the maps $\mu$ and $\nu$ are defined by (\ref{map}) and  (\ref{mapnu}), respectively.

\begin{proposition}\label{prop1} The maps  $\M_{a, \lam}$, $\N_{a, \lam}$, and their inverses preserve the space  $C^\infty(\sn)$ and are linear bounded operators on  $C(\sn)$ and $L^p(\sn)$, $1\le p\le \infty$.
\end{proposition}
\begin{proof}  The statements become obvious if we observe that
\[ 2\ge 1+a\tilde\eta_{n+1} \ge 1-|a|>0, \qquad 1\ge 1-a^2\xi_{n+1}^2\ge 1-a^2 >0,\]
because $0\!\le\! a\!<\!1$. The result for $L^p$ functions follows from  Lemma \ref{lem31}.\\
\end{proof}

\begin{lemma} \label{lem32} Let $\C^\lam$ and $\C_a^\lam$ be the $\lam$-cosine transforms  (\ref{cos}) and (\ref{cosa}), respectively.    If  $f\in L^1(\sn)$, then
\be\label{lopa} \C_a^\lam f= \N_{a, \lam}\C^\lam \M_{a, \lam} f.\ee
Moreover, $\C_a^\lam$ preserves the space  $C^\infty(\sn)$ and  is a linear bounded operator on  $C(\sn)$ and $L^p(\sn)$, $1\le p\le \infty$.
\end{lemma}
\begin{proof} We transform the left-hand side by making use of Lemma \ref{lem31}. Changing variables, for $\xi=\xi (\vp, s)$ and $\eta =\eta (\psi, u)$ we have
\bea
&&\xi\cdot (\eta -ae_{n+1})=\sqrt{1-s^2}\,\sqrt{1-u^2} \,(\vp\cdot \psi)+s(u-a)\nonumber\\
&&=\frac{\sqrt{1-s^2}\,\sqrt{1-a^2}\,\sqrt{1-v^2}}{1+av}\, (\vp\cdot \psi) +\frac{sv (1-a^2)}{1+av}\nonumber\\
&&=\frac{\sqrt{1-a^2}\,\sqrt{1-a^2s^2}}{1+av}\, \left [\frac {\sqrt{1-v^2}\,\sqrt{1-s^2}}{\sqrt{1-a^2s^2}}\,(\vp\cdot \psi) +
\frac{sv \sqrt{1-a^2}}{\sqrt{1-a^2s^2}}\right ]\nonumber\\
&&=\frac{\sqrt{1-a^2}\, \sqrt{1-a^2s^2}}{1+av}\, \left [\sqrt{1-v^2} \,\sqrt{1-\frac{s^2(1-a^2)}{1-a^2s^2}} \,(\vp\cdot \psi) +
v\frac{s\sqrt{1-a^2}}{\sqrt{1-a^2s^2}}\right ]\nonumber\\
&&=\frac{\sqrt{1-a^2}\,\sqrt{1-a^2s^2}}{1+av}\,(\tilde\xi \cdot \tilde\eta)= \frac{\sqrt{1-a^2}\,\sqrt{1-a^2\xi_{n+1}^2}}{1+a\tilde\eta_{n+1}}\,(\tilde\xi \cdot \tilde\eta).\nonumber\eea
Hence, by Lemma \ref{lem31},
\bea
&&\intl_{\sn} f(\eta) \,|\xi\cdot (\eta -ae_{n+1})|^\lam\, d\eta\nonumber\\
&&=(1-a^2)^{(n+\lam)/2}\, (1-a^2\xi_{n+1}^2)^{\lam/2}  \intl_{\sn} \frac{(f\circ \mu^{-1})(\tilde\eta)}
{(1+a\tilde\eta_{n+1})^{n+\lam}}\,|\tilde\xi \cdot \tilde\eta|^\lam \,  d\tilde\eta.\nonumber\eea
This gives (\ref{lopa}).  The last statement follows from the similar facts for $\C^\lam$ \cite[Section 5.1]{Ru15} in view of Proposition \ref{prop1}.
\end{proof}

Our next aim is to pass to the limit (\ref{lopa}) as $\lam \to -1$. This will give us the
 desired relation between the  slice transform $\F_a$ and the Funk transform $\F$.  We denote
\bea\label{ma}
\M_a f)(\tilde\eta)&=& \frac{(f\circ \mu^{-1})(\tilde\eta)}
{(1+a\tilde\eta_{n+1})^{n-1}}\equiv (\M_{a, \lam} f)(\tilde\eta)\big |_{\lam =-1}, \\
\label{na}(\N_a \Phi)(\xi)&=&\sig_{n-1}\, (1-a^2)^{(n-1)/2} (\Phi\circ \nu)(\xi)\\
&\equiv& \sig_{n-1}\, \sqrt {1-a^2\xi_{n+1}^2}\,  (\N_{a, \lam} \Phi)(\xi)\big |_{\lam =-1}.\nonumber\eea

\begin{theorem}\label {th1} If $f\in L^1(\sn)$, then
\be\label{fact}  \F_af =\N_a \F \M_a f.\ee
\end{theorem}
\begin{proof}  In view of Proposition \ref{prop1} and the $L^1$-boundedness of the Funk transform, it suffices to  prove (\ref{fact}) for $f\in C(\sn)$.
We write (\ref{lopa}) as $ (\C_a^\lam f)(\xi) = h_\lam (\xi) (\C^\lam g_\lam)(\tilde\xi)$, where
\[h_\lam (\xi)\!=\!(1-a^2)^{(n+\lam)/2} (1-a^2\xi_{n+1}^2)^{\lam/2}, \qquad   g_\lam (\tilde\eta)=\frac{(f\circ \mu^{-1})(\tilde\eta)}
{(1+a\tilde\eta_{n+1})^{n+\lam}}. \]
Here $\tilde\eta$ and $\tilde\xi$ are defined by (\ref{map}) and (\ref{mapnu}), respectively. 
Note that $g_\lam (\tilde\eta)$ is a continuous function of $(\lam, \tilde\eta)$ on $ [-1,0] \times \sn$ (we recall that $1+a\tilde\eta_{n+1} >0$ because $0\le a< 1$). If $\lam \to -1$, then
\[  h_\lam (\xi) \to  \frac{(1\!-\!a^2)^{(n-1)/2}}{\sqrt{1\!-\!a^2\xi^2_{n+1}}}, \qquad g_\lam (\tilde\eta) \to \frac{(f\circ \mu^{-1})(\tilde\eta)}
{(1+a\tilde\eta_{n+1})^{n-1}}=(\M_a f)(\tilde\eta).\]
Hence, by Lemma \ref{WWWLEM} in Appendix,
\[
\lim\limits_{\lam \to -1} (\C^\lam g_\lam)(\tilde\xi)=c_n \,(\F \M_a f) (\tilde\xi), \qquad  c_n= \frac{\pi^{1/2}}{\Gamma (n/2)},\]
where $F$ is the Funk transform (\ref{fu}).
Further, by  Lemma \ref{passlim2},
\[
\lim\limits_{\lam \to -1} (\C_a^\lam f)(\xi)\!= \!\frac{d_n}{\sqrt {1\!-\!a^2\xi_{n+1}^2}}  \, (\F_a f)(\xi),\quad d_n\!= \!\frac{\pi}{\Gamma ((n\!+\!1)/2)}\,. \]
Finally, (\ref{lopa}) yields
\[(\F_a f)(\xi)=  \sig_{n-1}\, (1-a^2)^{(n-1)/2} (\F\M_a f)(\tilde \xi)= (\N_a \F\M_a f)(\xi).\]
\end{proof}

Because the Funk transform is well-investigated \cite{GGG2, H11, P1, Ru15}, Theorem \ref{th1} enables us to establish the corresponding properties of $\F_a$. In particular,  Proposition \ref{prop1} yields the following corollary.
\begin{corollary} \label{cor1} The slice transform $\F_a$ preserves the space  $C^\infty(\sn)$ and  is a linear bounded operator on  $C(\sn)$ and $L^p(\sn)$, $1\le p\le \infty$.
\end{corollary}

It is known \cite{H11, Ru15} that the Funk transform $\F$, acting  on $L^1(\sn)$ or $C(\sn)$, is injective on even functions. Odd functions form the kernel of $\F$. Because the operators  $\M_a$ and $\N_a$ in (\ref{fact}) are bijective,  it follows that $\F_a$ is injective on the class of functions $f$ satisfying
\be\label{fiut}
(\M_a f)(\tilde\eta)=(\M_a f)(-\tilde\eta) \ee
for all (or almost all) $\tilde\eta \in \sn$. Similarly, the functions $f$ for which
\be\label{fiut1}
(\M_a f)(\tilde\eta)=-(\M_a f)(-\tilde\eta) \ee
form the kernel of $F_a$.

Below we characterize (\ref{fiut}) and (\ref{fiut1}) in terms of  $f$ itself.

Let, as (\ref{eta}) and (\ref {map}), $\eta\equiv \eta (\psi,u)$, $\tilde\eta =\mu (\eta) \equiv  \tilde\eta (\psi, v)$, where
\[\eta (\psi,u)=\sqrt {1-u^2} \,\psi + u e_{n+1}, \qquad \tilde\eta (\psi, v)= \sqrt {1-v^2} \,\psi + v e_{n+1},\]
\be\label{formx}  u=\frac{v+a}{1+av}, \qquad v=\frac{u-a}{1-ua}.\ee
By (\ref{ma})
\bea
(\M_a f)(\tilde\eta)&=&\frac{1}{(1+av)^{n-1}}(f\circ \mu^{-1})(\sqrt {1-v^2} \,\psi + v e_{n+1}),\nonumber\\
(\M_a f)(-\tilde\eta)&=&\frac{1}{(1-av)^{n-1}}(f\circ \mu^{-1})(-\sqrt {1-v^2}\, \psi - v e_{n+1}). \nonumber\eea
Using (\ref{formx}), it is straightforward to show that (\ref{fiut}) and (\ref{fiut1}) can be  equivalently written  as
\be\label{fon}
f(\eta)=\rho (\eta_*)\, f(\eta_*) \quad (\text{\rm or} \quad  f(\eta_*)=\rho (\eta)\, f(\eta))\ee
and
\be\label{fonb}
f(\eta)=-\rho (\eta_*)\, f(\eta_*) \quad (\text{\rm or} \quad  f(\eta_*)=-\rho (\eta)\, f(\eta)),\ee
respectively.
Here $\rho (\eta)\equiv \rho(\psi,u)$, $\rho (\eta_*)\equiv \rho (\psi_*,u_*)$, and $\eta_*\equiv \eta_*(\psi_*,  u_*)$ have the following definition in spherical coordinates:
\[\rho (\eta) =\left(\frac{1+a^2-2ua} {1-a^2}\right )^{n-1}, \quad \rho (\eta_*) =\left(\frac{1+a^2-2u_*a}{1-a^2}\right )^{n-1},\]
 \[ \eta_*(\psi_*,  u_*)= \sqrt {1-u_*^2} \,\psi_* + u_* e_{n+1},\quad \psi_*=-\psi, \quad
 u_* = \frac{2a - u(1+a^2)}{1+a^2-2ua}.\]
One can show  (we leave this calculation to the reader) that the points $\eta$, $\eta_*$, and  $A=(0, \ldots, 0,a)$ are located on the same line
 and the corresponding reflection map $\sn \ni \eta \to \eta_*\in \sn$  is an involution.

 The weighted equality (\ref{fon}) can be regarded as a  certain symmetry condition which is necessary and sufficient for injectivity of $\F_a$. In the Funk transform case $a=0$, (\ref{fon}) boils down to the classical evenness condition $f(\eta) =f(-\eta)$.

The above  reasoning can be summarized as follows.

\begin{theorem}\label {theo} The slice transform $\F_a$, acting on $L^1(\sn)$ or  $C(\sn)$,  is injective if and only if its action is restricted to  functions $f$ satisfying (\ref{fon}). The functions $f$ satisfying (\ref{fonb}) form the kernel of $\F_a$.
\end{theorem}

\begin{theorem}\label {th2} An integrable  function $f$ satisfying the symmetry condition (\ref{fon}) can be recovered from the  slice transform $\F_af$ by the formula
\be\label{recov}
f= \M_a^{-1} \F^{-1}\N_a^{-1}  \F_af\ee
where $\F^{-1}$ is the inverse Funk transform,
\bea
 &&(\M_a^{-1} \tilde f)(\eta)\!\equiv (\M_a^{-1} \tilde f)(\psi,u)\nonumber\\
\label{kiu} &&\!=\left (\frac{1\!-\!a^2}{1\!-\!ua}\right )^{n-1} \!\tilde f
\left (\frac{\sqrt{1\!-\!a^2} \, \sqrt{1\!-\!u^2}\, \psi +(u\!-\!a) e_{n+1}}{1\!-\!ua}\right ), \quad\eea
\bea
&&(\N_a^{-1} \tilde \Phi)(\tilde\xi)\!\equiv (\N_a^{-1} \tilde \Phi)(\vp,t)\nonumber\\
\label{kiu1} &&\!=\frac{(1\!-\!a^2)^{(1-n)/2}}{\sig_{n-1}} \tilde \Phi
\left (\frac{\sqrt{1\!-\!a^2} \, \sqrt{1\!-\!t^2}\, \vp \!+\!t e_{n+1}}{\sqrt{1\!-\!a^2\! +\!a^2t^2}}\right ). \quad\eea
\end{theorem}
\begin{proof}
The formula (\ref{kiu}) follows by routine calculation from (\ref{eta})-(\ref{rule}), (\ref{ma}). Similarly,
 (\ref{kiu1}) is a consequence of (\ref{mapnu})-(\ref{mapnu1}) and (\ref{na}).
\end{proof}

\begin{remark} \label{rara} The class of all functions $f$ that fall into the scope of Theorem \ref{th2} is non-empty and has the form $\M_a^{-1} (L^1_{even}(\sn))$, where the operator $\M_a^{-1}$ is defined by (\ref{kiu}). Similarly, $\M_a^{-1} (L^1_{odd}(\sn))$ is the kernel of $\F_a$ in $L^1(\sn)$.
\end{remark}

\begin{remark}  \label{rara1} Numerous  formulas for the inverse Funk transform $\F^{-1}$ in smooth and nonsmooth settings can be found in the literature; see, e.g., \cite[Section 5.1]{Ru15} and references therein.
 For example, the following statement  in terms of spherical means (\ref{sphm0}) is known (cf. \cite[Theorem 5.40]{Ru15}).
\begin{theorem}\label{invrhys} Let $f\in L^p_{even} (\sn)$, $1\le p<\infty$, $\Phi_\xi (s)=(M_{\xi} \,\F f)(\sqrt{1\!-\!s^2})$.
Then
\be\label{90ashel}
f(\xi) \!=  \! \lim\limits_{t\to 1}  \left (\frac {1}{2t}\,\frac {\partial}{\partial t}\right )^{n-1} \!\left [\frac{2}{(n-2)!}\intl_0^t
(t^2 \!- \!s^2)^{(n-3)/2} \,\Phi_\xi (s) \,s^{n-1}\,ds\right ].\qquad\qquad\ee
In particular, for $n$ odd,
\be\label{90ashele}
f(\xi) \!=  \! \lim\limits_{t\to 1} \frac{\pi^{1/2}}{\Gam (n/2)} \left (\frac {1}{2t}\,\frac {\partial}{\partial t}\right )^{(n-1)/2}[t^{n-2}\Phi_\xi (t)].\ee
The  limit in these formulas is understood in the $L^p$-norm. If $f\in C_{even}(\sn)$, it can be interpreted in the $\sup$-norm.
\end {theorem}
\end{remark}

\section{Spherical Slice Transforms via the Radon Transform}

In this section   the setting of the problem  is slightly changed and the tools are different.
Given  $-1< a \le 1$, let
 \be\label{sus} \sn_a=\{\eta \in \sn: \eta_{n+1}< a\}, \quad \sn_{-}\equiv  \sn_0=\{\eta \in \sn: \eta_{n+1}< 0\}.\ee
We consider a modified  spherical slice transform that takes a function $f$ on $\sn_a$
to a function $\S_a f$ on the set of all  spherical slices
\bea \gam_a (\xi)&=& \{\eta \in \sn_a: \xi \cdot (\eta -a e_{n+1})=0\} \nonumber\\
\label{sl} &=& \{\eta \in \sn_a: \xi \cdot \eta =a \xi_{n+1}\}, \eea
labeled by $\xi \in \sn_+$.
The value of $\S_a f$ at $\gam_a (\xi)$  is defined by the integral
\be\label{susmin} (\S_a f)(\xi)=\intl_{\gam_a (\xi)} f(\eta)\, d_{\gam_a}\eta, \ee
where $d_{\gam_a}\eta$ is the surface measure on $\gam_a (\xi)$. In the cases $a=0$ and $a=1$, we write $\S_a f$ as $\S_0 f$ and $\S_1 f$, respectively. The slices (\ref{sl}) are truncated versions of those in the previous section. They resemble multidimensional spherical ``arcs" that become complete spheres only if $a=1$. If $\xi \in \sn_+$, then
 \be\label{susgt}
(\S_a f)(\xi)= (\F_a \tilde f)(\xi),\ee
where $\tilde f (\eta)=f(\eta)$ if $\eta_{n+1}< a$ and  $\tilde f (\eta)=0$, otherwise.

\begin{proposition}\label {kshm}
If $-1<  a <1$,  then $\S_a$ is a linear bounded operator from  $L^p(\sn_a)$ to $L^p(\sn_+)$ for all  $1\le p\le \infty$.
\end{proposition}
\begin{proof}  By (\ref{susgt}) and Corollary \ref {cor1},
\[
||\S_a f||_{L^p(\sn_+)}=||\F_a \tilde f||_{L^p(\sn)}\le c\, ||\tilde f||_{L^p(\sn)}= c\, ||f||_{L^p(\sn_a)}.\]
\end{proof}

\subsection{Stereographic Projection for  Truncated Sphere}
We recall that the space $\rn$ is identified with the coordinate hyperplane $e_{n+1}^\perp$ in $\rnp1$. Consider the bijective map
\be\label{mapx} \rn \ni x \xrightarrow{\;\pi\;} \eta \in \sn_a,  \ee
where the point $\eta=\pi (x)$ is located on the line $\ell$ passing through  the point  $A=(0, \ldots, 0,a)$ and the point $X=x-e_{n+1}$ lies in the tangent plane to $\sn$ at the South Pole $p_S=(0, \ldots, 0, -1)$. To find $\eta=\pi (x)$ explicitly, we  make use of the parametric representation of $\ell$, so that
 \[\ell=\{y \in \rnp1: \;y=tA+(1-t)X; \quad t\in \bbr \},\] and find the value of $t$ for which $y\in\sn_a$. Let
\[
\eta =\psi \sin \om +e_{n+1}  \cos \,\om, \qquad \psi\in \snm1, \quad 0\le \om\le \pi.\]
Then, for $y =(y', y_{n+1})\in \rnp1$ we have
\[\ell=\{y: \;y' =(1-t)x, \quad y_{n+1}=ta+t-1;\quad  t\in \bbr \}.\]
The condition  $|y|=1$ with $y\in\sn_a$ yields the quadratic equation
\[ (1-t)^2|x|^2 + (ta+t-1)^2 =1, \qquad t<0.\]
The solution to this equation has the form
\[ t=\frac{|x|^2+a+1-D(x)}{|x|^2+(a+1)^2}, \]
where
\be\label {pmuv} D(x)=\sqrt{a+1}\,\sqrt{|x|^2 (1-a) +a+1}.\ee
This gives
\be\label {pq} \eta\equiv\pi(x)=P(x)\, x+Q(x)\, e_{n+1}\in\sn_a, \ee
where
\be\label {pqw} P(x)=\frac{a(a+1)+D(x)}{|x|^2 +(a+1)^2}, \qquad Q(x)=\frac{a|x|^2 - (a+1)D(x)}{|x|^2 +(a+1)^2}.\ee
In particular,
\bea\label{mytrin} &&\text {\rm for $a=0$:} \qquad D(x)=\sqrt{|x|^2 +1}, \qquad P(x)=\frac{1}{\sqrt{|x|^2 +1}};\qquad\\
\label{mytrin1} &&\text {\rm for $a=1$:} \qquad D(x)=2, \qquad P(x)=\frac{4}{\sqrt{|x|^2 +4}}.\qquad\eea

The inverse map
\be\label{mapxin} \sn_a \ni \eta  \xrightarrow{\;\pi^{-1}\;} x \in \rn,  \ee
has  much simpler analytic expression that follows from elementary geometrical consideration. Specifically, if $\eta=(\eta', \eta_{n+1})$, $\eta'=(\eta_1, \ldots, \eta_n)$, or $\eta= \psi \sin \om +e_{n+1}  \cos \,\om$, $\psi \in \snm1$,  $r=|x|$, then
\be\label {pq44} |x|=\frac{(a+1)\, \sqrt {1-\eta_{n+1}^2}}{a -\eta_{n+1}} =\frac{(a+1)\, \sin \om}{a -\cos\, \om},\ee
and therefore
\be\label{mapxinf} x=\pi^{-1} (\eta)=\frac{(a+1)\,\eta'}{a -\eta_{n+1}}.\ee
In particular, (\ref {pq}) and (\ref {pq44}) yield
\be\label{ntin}
P(\pi^{-1} (\eta))=\frac {a -\eta_{n+1}}{a+1}, \qquad Q(\pi^{-1} (\eta))=\eta_{n+1}.\ee

\begin{remark}\label{knaq}  In the cases $a=0$ and $a=1$, of slices through the origin and the North Pole $p_N$,  we have  the following simple formulas:

\vskip 0.2 truecm

\noindent {\it  For $a=0$}:
\be\label {pqer}
\eta =\frac{x-e_{n+1}}{\sqrt{|x|^2 +1}}= \frac{x-e_{n+1}}{|x-e_{n+1}|};  \qquad x\in \rn, \ee
\be\label {pqernj} x=\frac{\eta'}{|\eta_{n+1}|}, \qquad \eta \in \sn_{-}.\ee
\noindent {\it For $a=1$}:
\be\label {pqer1}
\eta =\frac{4x  + (|x|^2 -4) e_{n+1} }{|x|^2 +4},  \qquad x\in \rn; \ee
\be\label {pqergt1} x=\frac{2\eta'}{1-\eta_{n+1}}, \qquad \eta \in \sn\setminus p_N.\ee
\end{remark}

\begin{lemma}\label{Loi} Let $-1< a\le 1$,
\be\label {pq1tr}
w (\eta)=\frac{(a+1)^n (1-a \eta_{n+1})}{(a-\eta_{n+1})^{n+1}}, \qquad \eta \in \sn_a,\ee
so that
\be\label {psq1tr}
(w \circ\pi)(x)=\frac{D(x)}{(a+1)\, P^n (x)},\qquad  x\in \rn,\ee
where $D(x)$ and $P(x)$ are defined by (\ref{pmuv}) and (\ref{pqw}), respectively.
Then for any $g\in L^1(\rn)$ and $f\in L^1 (\sn_a)$,
\be\label {pq2}
\intl_{\rn} g(x)\, dx=\intl_{\sn_a} (g\circ \pi^{-1})(\eta)\, w (\eta)\, d\eta,\ee
\be\label {pq3}
\intl_{\sn_a} f(\eta)\, d\eta =\intl_{\rn} (f \circ \pi)(x)\, \frac{dx}{(w \circ\pi)(x)}.\ee
 \end{lemma}
\begin{proof} We observe that (\ref{psq1tr}) follows from (\ref{pq1tr}) by virtue of (\ref{ntin}). Further, passing to polar coordinates and using (\ref{pq44}), we obtain
\bea
&&\intl_{\rn} g(x)\, dx=\intl_0^\infty r^{n-1}\, dr \intl_{\snm1} g(r\th) \, d\th \nonumber\\
&&=\intl_{\cos^{-1} a}^\pi \!\!\left (\frac{(a\!+\!1)\, \sin \om}{a \!-\!\cos\, \om} \right )^{n-1} \frac{(a\!+\!1)\,(1\!-\!\cos\, \om)}{(a \!-\!\cos\, \om)^2}\, d\om \nonumber\\
&& \times
\intl_{\snm1} \!\!g\left (\frac{(a\!+\!1)\, \sin \om}{a \!-\!\cos\, \om} \, \th\right ) d \th,\nonumber\eea
which is  the right-hand side of (\ref{pq2})  in spherical coordinates. The equality (\ref{pq3})  follows from (\ref{pq2}) if we set $f(\eta)=(g\circ \pi^{-1})(\eta)\, w (\eta)$.
\end{proof}

\begin{example}  In the simplest cases, $a=0$ and $a=1$, the formulas (\ref{pq2}) and  (\ref{pq3}) have the following form.

\vskip 0.2 truecm

\noindent {\it  For $a=0$}:
\be\label {pq20}
\intl_{\rn} g(x)\, dx=\intl_{\sn_-} g\left (\frac{\eta'}{|\eta_{n+1}|}\right )\, \frac{d\eta}{|\eta_{n+1}|^{n+1}},\ee
\be\label {pq30}
\intl_{\sn_{-}} f(\eta)\, d\eta =\intl_{\rn} f\left (\frac{x-e_{n+1}}{\sqrt{|x|^2 +1}}\right ) \frac{dx}{(|x|^2 +1)^{(n+1)/2}}.\ee

\noindent {\it  For $a=1$}:

\be\label {pq201}
\intl_{\rn} g(x)\, dx=2^n\intl_{\sn} g\left (\frac{2\eta'}{1-\eta_{n+1}}\right )\, \frac{d\eta}{(1-\eta_{n+1})^{n}},\ee
\be\label {pq301}
\intl_{\sn} f(\eta)\, d\eta =4^n\intl_{\rn} f\left (\frac{4x  + (|x|^2 -4) e_{n+1} }{|x|^2 +4}\right ) \frac{dx}{(|x|^2 +4)^{n}}.\ee

\end{example}

\subsection{Semyanistyi's Integrals, Cosine, and Radon Transforms} ${}$\\

{\bf 1.} The map (\ref{mapx}) extends to the corresponding relationship between affine hyperplanes in $\rn$ and spherical slices (\ref{sl}). Specifically, to every $(n-1)$-plane $\t$ in $\rn=e_{n+1}^\perp$ we associate a spherical slice in the $n$-dimensional plane
spanned by the shifted $(n-1)$-plane
$\t\! -\!e_{n+1}$ and the point  $A=(0, \ldots, 0,a)$. This relationship gives rise to the corresponding connection between the hyperplane Radon transform  of functions on $\rn$ and the spherical slice transform  (\ref{susmin}).

Before we describe  this connection analytically, it is instructive to recall some basic facts from related to the Radon transform and the corresponding analytic family of Semyanistyi's integrals; see, e.g., \cite{H11, Ru15}.

The  hyperplane Radon transform $R$  takes a sufficiently good continuous function $g$ on $\rn$ to integrals of $g$ over affine hyperplanes in $\rn$. Every such hyperplane  $\t$  can be parameterized as
\[\t \equiv \t (\th, t)=\{x\in \rn : x\cdot \th =t\}, \qquad  \th \in \snm1, \qquad t\in \bbr, \]
 so that $\t (\th, t)=\t(-\th, -t)$. In this notation,
\be\label{rad} (Rg) (\t)\equiv (Rg) (\theta, t)=\intl_{\theta^{\perp}} g(t\theta +
u) \,d_\theta u,\ee where $d_\theta u$ denotes the Euclidean measure on $\theta^{\perp}$.

If $C_c^\infty (\rn)$ is the space of compactly supported infinitely differentiable functions on $\rn$ and $C_c^\infty (\snm1 \times \bbr)$ is the similar space on $ \snm1 \times \bbr$, then  $R$ acts from $C_c^\infty (\rn)$ to $C_c^\infty (\snm1 \times \bbr)$. A similar statement holds for the spaces $C_c (\rn)$ and $C_c (\snm1 \times \bbr)$ of compactly supported continuous functions.

The integral (\ref{rad}) is well-defined for almost all $(\th, t)\in \snm1 \times \bbr$ whenever $g\in L^p (\rn)$, $1\le p <n/(n-1)$, and these bounds  are sharp. Moreover, if $\snm1 \times \bbr$ is equipped with the product measure $d_* \th dt$ then
\be\label{duas3} \intl_{\snm1 \times \bbr} \frac{(Rg)(\theta,
t)}{(1+t^2)^{n/2}}\,d_*\theta dt= \intl_{\bbr^n}
\frac{g(x)}{(1+|x|^2)^{1/2}}\,dx;
 \ee
see \cite[formula (4.4.4)]{Ru15}. It is also known  \cite{OS}   that  $R$ is a linear bounded operator from $L^p (\rn)$ to $L^q (\snm1 \times \bbr)$ if $1/q=n/p -n +1$, $1\le p <n/(n-1)$.   If $g(x)\equiv g_0 (|x|)$ is
 a  radial function, then $(Rg) (\theta, t)$ is an even function of only one variable $t$. If we denote this function by $G(t)$, then
\be\label{fouyg}
G(t)=\sigma_{n-2} \intl^\infty_{|t|}\! g_0 (r)
(r^2-t^2)^{(n-3)/2}\,r dr;\ee
cf. \cite[Lemma 4.17]{Ru15}.

 The Funk transform and the Radon transform express one through another by making use of the stereographic projection; see \cite[p. 317]{Ru15}. This connection allows us to reformulate  known properties of one operator in terms of another. In this section obtain the corresponding results for more general spherical slice transforms.

The  hyperplane Radon transform can be regarded as a member of the analytic family of the
{\it  Semyanistyi fractional integrals} depending on a complex parameter $\lam$ and defined by the formula
\be \label{sosa}  (R^\lam g)(\th, t) =\gam_{n,\lam}
 \intl_{\rn} g(x) |x \cdot \th -t|^{\lam} \,dx, \qquad Re \, \lam > -1, \ee
where $\gam_{n,\lam}$ is the constant (\ref{siefu}). If $g$ is good enough, then
\be \label{rrrr}
\lim\limits_{\lam \to -1} (R^\lam g)(\th, t) = d_n(Rg) (\theta, t), \quad  d_n\!=\! \frac{\pi}{\Gamma ((n\!+\!1)/2)}.\ee
This equality (in different notation)  is well known (see, e.g., \cite[Section 4.9]{Ru15}). It will be proved in Appendix (see Lemma \ref {WWWny}) for   more general integrals, when $g$ depends also on $\lam$.

{\bf 2.} We introduce  an analytic family of
 cosine transforms associated with slices in $\sn_a$ by the formula
\be \label{cosan} (\frC_a^\lam f)(\xi) =\gam_{n,\lam}
 \intl_{\sn_a} f(\eta)\, |\xi \cdot \eta -a\xi_{n+1}|^{\lam} \,d\eta, \quad \xi \in \sn_+, \ee
(cf. (\ref{cosa})). For $x\in \rn$ and $\xi \!\equiv  \! \xi(\th, s)\!=\!\sqrt {1-s^2}\,\th \!+\!se_{n+1}\! \in \!\sn_+\setminus p_N$, $\th \in \snm1$, $s\in [0,1)$, we
set
\bea \label{hyan} (\U_{a, \lam} f)(x)\!\!&=&\!\!\frac{(f \circ \pi)(x)}{(w \circ \pi)(x)}\, P^\lam (x), \\
\label{hyan2} (\V_{a, \lam} \Phi)(\xi)\!\!&\equiv& \!\!(\V_{a, \lam} \Phi)(\th, s)\!=\!(1\!-\!s^2)^{\lam/2} \Phi \left (\th, \frac{(a+1)s}{\sqrt {1-s^2}}\right ),\quad \eea
 where  $P(x)$ is defined by (\ref {pqw}), and $(w \circ\pi)(x)$ has the form (\ref{psq1tr}).

\begin{example}
In the cases $a=0$ and $a=1$, (\ref{hyan}) and (\ref{hyan2})  have the following form (cf. Remark \ref{knaq}).

\vskip 0.2truecm
\noindent {\it For $a=0$}:
\bea \label{hyan0} (\U_{0, \lam} f)(x)\!&=&\!\frac{1}{(|x|^2 +1)^{(\lam +n+1)/2}}\, f\left( \frac{x-e_{n+1}}{\sqrt{|x|^2 +1}} \right), \\
\label{hyan20} (\V_{0, \lam} \Phi)(\xi)\!&\equiv& \!(\V_{0, \lam} \Phi)(\th, s)\!=\!(1\!-\!s^2)^{\lam/2} \Phi \left (\th, \frac{s}{\sqrt {1-s^2}}\right ).\quad \eea

\noindent {\it For $a=1$}:
\bea \label{hyan1} (\U_{1, \lam} f)(x)\!&=&\!\left (\frac{4}{|x|^2 +4}\right )^{n+\lam}\, f\left( \frac{4x  + (|x|^2 -4) e_{n+1} }{|x|^2 +4} \right), \\
\label{hyan21} (\V_{1, \lam} \Phi)(\xi)\!&\equiv&\! (\V_{1, \lam} \Phi)(\th, s)\!=\!(1\!-\!s^2)^{\lam/2} \Phi \left (\th, \frac{2s}{\sqrt {1-s^2}}\right ).\quad \eea
\end{example}

Let $C_c (\sn_a)$ be the space of continuous functions on $\sn$ having compact support in $\sn_a $. Such functions vanish identically near the parallel $\eta_{n+1}=a$.  The corresponding space of infinitely differentiable functions  is denoted by $C_c^\infty (\sn_a)$.
The  notations $C_c (\sn_+ \setminus p_N)$  and  $C_c^\infty (\sn_+ \setminus p_N)$ will be used for the similar spaces of  functions  on $\sn_+ \setminus p_N$.

The definitions (\ref{hyan}) and (\ref{hyan2})
imply the following statements.

\begin{proposition}\label{prop1s} ${}$\\

\noindent {\rm (i)}
The operator  $\U_{a, \lam}$ acts bijectively from $C_c (\sn_a)$ {\rm(}from $C_c^\infty (\sn_a)${\rm)} to $C_c (\rn)$ {\rm(}to $C_c^\infty (\rn)${\rm)}. The inverse operator  $\U_{a, \lam}^{-1}$ acts similarly in the opposite direction.

\noindent {\rm (ii)}
The operator  $\V_{a, \lam}$ acts bijectively from $C_c (\snm1 \times \bbr)$ {\rm(}from $C_c^\infty (\snm1 \times \bbr)${\rm)}
to $C_c (\sn_+ \setminus p_N)$ {\rm(}to $C_c^\infty (\sn_+ \setminus p_N)${\rm)}. The inverse operator  $\V_{a, \lam}^{-1}$ acts similarly in the opposite direction.
\end{proposition}

\begin{lemma}\label{kiaw} Let $f\in C_c (\sn_a)$, $-1< a\le 1$,
\[\xi=\sqrt {1-s^2}\,\th +se_{n+1} \in \sn_+ \setminus p_N, \qquad \th \in \snm1, \quad s\in [0,1).\]
Then
\be\label{law}
(\frC_a^\lam f)(\xi)\equiv  (\frC_a^\lam f)(\th, s) = (\V_{a, \lam} R^\lam \U_{a, \lam} f)(\th, s),\ee
\end{lemma}
\begin{proof} By (\ref{pq3}),
\[
(\frC_a^\lam f)(\xi)=\gam_{n,\lam} \intl_{\rn} (f \circ \pi)(x)\,   |\xi \cdot \pi (x) -a\xi_{n+1}|^{\lam} \,\frac{dx}{(w \circ \pi)(x)}.\]
Further, by (\ref {pq}),
\bea
 &&\xi \cdot \pi (x) -a\xi_{n+1}=  \xi \cdot (\pi (x) -ae_{n+1})\nonumber\\
&&=(\sqrt {1-s^2}\th +se_{n+1})\cdot  [P(x)\,x+(Q(x)-a)e_{n+1}]\nonumber\\
&&=P(x) \sqrt {1-s^2} \left [x \cdot \th +\frac{(Q(x)-a)s}{P(x)\,\sqrt {1-s^2}}\right ]=P(x)\,\sqrt {1-s^2}\,(x \cdot \th -t),\nonumber\eea
where
\be\label{y34trz}   t=\frac{(a-Q(x))s}{P(x)\,\sqrt {1-s^2}}=\frac{(a+1)s}{\sqrt {1-s^2}}.\ee
The last expression follows from (\ref {pqw}). Thus,
\bea (\frC_a^\lam f)(\xi)&=&\gam_{n,\lam} \,(1-s^2)^{\lam/2} \intl_{\rn} (f \circ \pi)(x)\,   |P(x)\,(x \cdot \th -t)|^{\lam} \,\frac{dx}{(w \circ \pi)(x)}\nonumber\\
&=& \gam_{n,\lam}  \,(1-s^2)^{\lam/2} \intl_{\rn} (\U_{a, \lam} f)(x)|x \cdot \th -t)|^{\lam} \,dx,\nonumber\eea
as desired; cf. (\ref{hyan2}).
\end{proof}

Our next aim is to pass to the limit in (\ref{law}) as $\lam \to -1$.  For $x\in \rn$ and $\xi\equiv \xi (\th, s)\in \sn_+$, we denote
\bea \label{hyanb} (\U_{a} f)(x)\!\!\!&=&\!\!\!\frac{(f \circ \pi)(x)}{(w \circ \pi)(x)\, P(x)}=\frac{(a+1) \,P^{n-1}(x)}{D(x)}\, (f \circ \pi)(x), \quad\\
\label{hyan2b} (\V_{a} \Phi)(\xi)\!\!\!&=&\!\!\!\sqrt{\frac{1\!-\!a^2s^2}{1\!-\!s^2}} \Phi \left (\th, \frac{(a+1)s}{\sqrt {1-s^2}}\right ), \eea
where  $P(x)$ and $D(x)$ are  defined by (\ref {pqw}) and (\ref{pmuv}), respectively, and $(w \circ\pi)(x)$ has the form (\ref{psq1tr}). In particular,

\vskip 0.2truecm
\noindent {\it for $a=0$}:
\bea \label{xhyan0} (\U_{0} f)(x)\!&=&\!\frac{1}{(|x|^2 +1)^{n/2}}\, f\left( \frac{x-e_{n+1}}{\sqrt{|x|^2 +1}} \right), \\
\label{xhyan20} (\V_{0} \Phi)(\xi)\!&=&  \sqrt{\frac{1}{1\!-\!s^2}} \Phi \left (\th, \frac{s}{\sqrt {1-s^2}}\right ); \eea

\noindent {\it for $a=1$}:
\bea \label{xhyan1} (\U_{1} f)(x)\!&=&\!\left (\frac{4}{|x|^2 +4}\right )^{n-1}\, f\left( \frac{4x  + (|x|^2 -4) e_{n+1} }{|x|^2 +4} \right), \\
\label{xhyan21} (\V_{1} \Phi)(\xi)\!&=&\! \Phi \left (\th, \frac{2s}{\sqrt {1-s^2}}\right ). \eea

The analytic expressions for the inverse operators can be easily obtained by making use of (\ref{mapxinf}), (\ref{ntin}), (\ref{pq1tr}), and (\ref{y34trz}). Specifically,

\be\label{ytrz}
(\U_{a}^{-1} \tilde f)(\eta)=\frac{(a+1)^{n-1} (1-a \eta_{n+1})}{(a-\eta_{n+1})^{n}}\, \tilde f \left ( \frac{(a+1)\,\eta'}{a -\eta_{n+1}}\right ), \ee
\be\label{ytrz1}
(\V_{a}^{-1} \tilde \Phi)(\theta, t)= \sqrt \frac{a+1}{t^2 (1-a)+a+1} \,\tilde \Phi \left ( \frac{(a+1)\,\theta +te_{n+1}}{\sqrt{t^2 +(a+1)^2}} \right ). \ee
In particular,

\vskip 0.2truecm
\noindent {\it for $a=0$}:
\bea \label{ffnw21} (\U_{0}^{-1} \tilde f)(\eta)&=&\frac{1}{|\eta_{n+1}|^{n}}\, \tilde f \left ( \frac{\eta'}{|\eta_{n+1}|}\right ), \\
\label{ffan21} (\V_{0}^{-1} \tilde \Phi)(\theta, t)&=& \frac{1}{\sqrt {t^2+1}} \,\tilde \Phi \left ( \frac{\theta +te_{n+1}}{\sqrt{t^2 +1}} \right ); \eea

\noindent {\it for $a=1$}:
\bea \label{ffn21b} (\U_{1}^{-1} \tilde f)(\eta)&=&\left (\frac{2}{1-\eta_{n+1}}\right )^{n-1}\, \tilde f \left ( \frac{2\eta'}{1 -\eta_{n+1}}\right ),  \\
\label{ffn21} (\V_{1}^{-1} \tilde \Phi)(\theta, t)&=& \tilde \Phi \left ( \frac{2\theta +te_{n+1}}{\sqrt{t^2 +4}} \right ).\eea

\begin{lemma} \label {jaau} If $f \in C_c(\sn_a)$, $-1< a\le 1$, then
\be\label{brv}
(\S_a f)(\xi) =   (\V_a R\,\U_{a} f)(\xi) \qquad \forall \,\xi \in \sn_+\setminus p_N.\ee
\end{lemma}
\begin{proof}  We first show that
\be\label{lims}
\lim\limits_{\lam \to -1} (\frC_a^\lam f)(\xi)\!=\!\frac{ d_n\, (\S_a f)(\xi)}{\sqrt{1- a^2\xi_{n+1}^2}},   \quad d_n\!=\! \frac{\pi}{\Gamma ((n\!+\!1)/2)}.
\ee
To this end, we make use of Lemma \ref{passlim2}, according to which
\[
\lim\limits_{\lam \to -1} (\C_a^\lam f)(\xi)=\frac{1}{\sig_n}\lim\limits_{\lam \to -1} (\frC_a^\lam f)(\xi)=c_n\, (\F_a f)(\xi) (1-a^2\xi_{n+1}^2)^{(n-2)/2}.\]
By (\ref {fuat}) and  (\ref{sphm01}),
\[
(\F_a f)(\xi)=(M_{\xi} f)(a\xi_{n+1})=\frac{(1-a^2\xi_{n+1}^2)^{(1-n)/2}}{\sig_{n-1}} \intl_{\xi \cdot \eta =a\xi_{n+1}} f(\eta)\, d\sig(\eta). \]
Hence
\[
\lim\limits_{\lam \to -1} (\frC_a^\lam f)(\xi)=\frac{c_n\, \sig_n}{\sig_{n-1}} (1-a^2\xi_{n+1}^2)^{-1/2} (\S_a f)(\xi),\]
which coincides with (\ref{lims}).

To find the limit of the right-hand side of (\ref{law}) as $\lam \to -1$, we make use of the properties of the
 Semyanistyi type integral in Lemma \ref{WWWny}. The function  $\rho (\a, x)=(\U_{a, \lam} f)(x)\big |_{\lam =\a -1}$ satisfies the conditions of this lemma, and we get
\be\label{ltms}
\lim\limits_{\lam \to -1} (R^\lam \U_{a, \lam} f)(\th, t)= d_n (R\,\U_{a, -1} f)(\th, t). \ee
Hence, for  the right-hand side of (\ref{law}) we have
\be\label{dltms}
\lim\limits_{\lam \to -1}  (\V_{a, \lam} R^\lam \U_{a, \lam} f)(\th, s)=\frac{d_n}{\sqrt{1\!-\!s^2}} (R\,\U_{a, -1} f)
 \left (\th, \frac{(a+1)s}{\sqrt {1-s^2}}\right ).\ee
Comparing (\ref{lims}) with (\ref{dltms}), we obtain the result.
\end{proof}

To extend the factorization (\ref{brv}) to Lebesgue integrable functions and  obtain the corresponding inversion formula,
 some auxiliary estimates  will be needed. We recall the notation from (\ref{pq3}):
\be\label {pqwzx} P(x)\!=\!\frac{a(a\!+\!1)\!+\!D(x)}{|x|^2\! +\!(a\!+\!1)^2}, \quad D(x)\!=\!\sqrt{a\!+\!1}\,\sqrt{|x|^2 (1\!-\!a)\! +\!a\!+\!1}.\ee

\begin{lemma} \label {jqw1} Let $1\le p<\infty$,  $-1< a\le 1$, $\lam \in \bbr$. Then
\bea\label{nr1v}
&&\intl_{\rn} |(\U_{a} f)(x)|^p \,P^\lam (x)\, dx\\
&&=(a\!+\!1)^{p-\lam +n(1-p)}\intl_{\sn_a}\frac{|f(\eta)|^p \,(a\!-\!\eta_{n+1})^{\lam -1+n(p-1)}} {(1\!-\!a\eta_{n+1})^{p-1}}\, d\eta\nonumber\eea
provided that the integral in either side exists in the Lebesgue sense.
\end{lemma}
\begin{proof} By (\ref{hyanb}) and (\ref{pq2}),
\[
l.h.s.=\!\intl_{\rn}  \left |\frac{(f \circ \pi)(x)}{(w \circ \pi)(x)\, P(x)}\right |^p P^\lam (x)\, dx=
\!\intl_{{\sn_a}} \!\frac{|f(\eta)|^p}{[P(\pi^{-1}(\eta)]^{p-\lam} \, \,[w (\eta)]^{p-1}}\, d\eta.\]
Hence, by (\ref{ntin}) and (\ref{pq1tr}),
\[
l.h.s.=\!\intl_{{\sn_a}} \!|f(\eta)|^p \left (\frac {a -\eta_{n+1}}{a+1}\right )^{\lam -p}  \left [\frac{(a+1)^n (1-a \eta_{n+1})}{(a-\eta_{n+1})^{n+1}}\right ]^{1-p} d\eta=r.h.s.\]
\end{proof}

The following corollary contains important particular cases of (\ref{nr1v})  that will be used in the following. Here we make use of the formulas (\ref{mytrin})   and  (\ref{mytrin1}). As above, we assume that the integral in either side of the corresponding equality exists in the Lebesgue sense.

\begin{corollary}\label{dlfluig} ${}$ \hfill

\vskip 0.2 truecm

\noindent $\bullet$ $\; \lam =0$, $1\le p<\infty$, {\rm  $-1< a\le 1$:}
\be\label{bcddrv}
\intl_{\rn} |(\U_{a} f)(x)|^p \, dx\!=\!(a\!+\!1)^{n+p(1-n)}\intl_{\sn_a}\frac{|f(\eta)|^p \, d\eta}{(a\!-\!\eta_{n+1})^{n+1-np}\, (1\!-\!a\eta_{n+1})^{p-1}};\ee

\noindent $\bullet$  $\; \lam =0$, $1\le p<\infty$,  {\rm $a= 1$:}
\be\label{bcddrv1}
\intl_{\rn} |(\U_{1} f)(x)|^p \, dx=2^{n+p(1-n)}\intl_{\sn}\frac{|f(\eta)|^p \, d\eta}{(1-\eta_{n+1})^{n-p(n-1)}};\ee

\noindent $\bullet$  $\; \lam =0$, $1\le p<\infty$, {\rm $a= 0$:}
\be\label{bcddrv2}
\intl_{\rn} |(\U_{0} f)(x)|^p \, dx=\intl_{\sn_-}\frac{|f(\eta)|^p \, d\eta}{|\eta_{n+1}|^{n+1-np}};\ee

\noindent $\bullet$  $\; \lam =1$, $p=1$, {\rm $-1< a\le 1$:}
\be\label{fytrz} \intl_{\rn} (\U_{a} f)(x)\, P(x)\, dx =\intl_{\sn_a} f (\eta)\, d\eta;\ee

\noindent $\bullet$  $\; \lam =1$, $p=1$, {\rm $ a=1$:}
\be\label{fhrz} \intl_{\rn} \frac{(\U_{1} f)(x)}{|x|^2 +4}  \,dx=\frac{1}{4}\intl_{\sn} f (\eta)\, d\eta;\ee

\noindent $\bullet$  $\; \lam =1/2$, $p=1$, {\rm $ a=1$:}
\be\label{fhrz} \intl_{\rn} \frac{(\U_{1} f)(x)}{\sqrt{|x|^2 +4}}  \,dx=\frac{1}{\sqrt {2}}\intl_{\sn} \frac{f (\eta)}{\sqrt {1-\eta_{n+1}}}\, d\eta.\ee
\end{corollary}

\begin{corollary}\label{dlfluig1}
If $-1< a< 1$, then there is a constant $c_a>0$ such that
\bea\label{brrv}
\intl_{\rn} |(\U_{a} f)(x)|^p \, dx &\le& c_a \intl_{\sn_a}\frac{|f(\eta)|^p \, d\eta}{(a-\eta_{n+1})^{n+1-np}}, \\
\label{fyqatrz} \intl_{\rn}\frac{ |(\U_{a} f)(x)|}{(1+|x|^2)^{1/2}}\, dx &\le&  c_a\intl_{\sn_a} |f (\eta)|\, d\eta.\eea
\end{corollary}
\begin{proof} These  inequalities follow from (\ref{bcddrv}) and (\ref{fytrz}). Here we keep in mind that  $P(x)$ essentially behaves like $ (1+|x|^2)^{-1/2}$ and $1\!-\!a\eta_{n+1}$ is separated from zero,  because $0\le a< 1$.
\end{proof}

{\bf 3.} Now we are ready to prove main theorems of this section.

\begin{theorem} \label {jaau1} If $-1< a < 1$ and $f \in L^1(\sn_a)$,  then
\be\label{brv}
(\S_a f)(\xi) =   (\V_a R\,\U_{a} f)(\xi)\quad \text{for almost all}\quad \xi\in\sn_+.\ee
\end{theorem}
\begin{proof}  Fix any  $\e\in (0, 1)$ and denote  $\bbs^{n, \e}_+=\{\xi \in\sn_+: \, 0\le \xi_{n+1}< 1-\e\}$.   Let us show  that  $\S_a$ and  $\V_a R\,\U_{a}$ are linear bounded operators from   $L^1(\sn_a)$ to $L^1(\bbs^{n, \e}_+)$ . The estimate
\be\label{bfeggrv}
||\S_a f||_{L^1(\bbs^{n, \e}_+)} \le  c\, ||f||_{L^1(\sn_a)}.\ee
 holds by Proposition \ref{kshm}. Further, we observe that there exist positive constants  $c_\e$ and $a_\e$, depending only on $\e$ and $a$, such that
\be\label{dwss}
I_\e \equiv\intl_{\bbs^{n, \e}_+} |(\V_a \Phi)(\xi)|\, d\xi \le c_\e \intl_{0}^{a_\e} dt \intl_{\snm1} |\Phi (\th, t)|\, d\th\ee
for any function $\Phi(\th, t)$ on $\snm1 \times \bbr$ whenever the right-hand side of this inequality is finite. Indeed, by (\ref{cnhos}) and
 (\ref{hyan2b}),
\bea
I_\e&=&\intl_{0}^{1-\e} (1\!-\!s^2)^{(n-2)/2}\, ds\intl_{\snm1}|(\V_a \Phi)(\sqrt{1\!-\!s^2} \, \th +se_{n+1})|\, d\th\nonumber\\
&=&\intl_{0}^{1-\e} (1\!-\!s^2)^{(n-2)/2}\,\sqrt{\frac{1\!-\!a^2s^2}{1\!-\!s^2}} \, ds\intl_{\snm1}\left |\Phi \left (\th, \frac{(a+1)s}{\sqrt {1-s^2}}\right )\right |\, d\th.\nonumber\eea
Changing variable $t=(a+1)s/\sqrt {1-s^2}$, so that $ s=t/\sqrt {t^2 +(a+1)^2}$, and setting
\[
a_\e= \frac{(a+1) (1-\e)}{\sqrt{\e}\, \sqrt{2-\e}},\]
we continue
\bea I_\e&=&\intl_{0}^{a_\e} \left [\frac{(a+1)^2}{t^2 +(a+1)^2}\right]^{(n-2)/2}\,\sqrt {\frac{t^2(1-a) + a+1}{a+1}}\nonumber\\
&\times& \frac{(a+1)^2\, dt}{[t^2 +(a+1)^2]^{3/2}} \intl_{\snm1}|\Phi (\th,t)|\,d\th.\nonumber\eea
Now the routine estimate  gives (\ref{dwss}).

If $\Phi =R\,\U_a f$, then  (\ref{dwss}) yields
\bea
I_\e &\le& c_\e \intl_{0}^{a_\e} dt \intl_{\snm1} |(R\,\U_a f) (\th, t)|\, d\th\le \tilde c_\e \intl_{0}^{a_\e} dt
\intl_{\snm1} \frac{(R\,\U_a |f|) (\th, t)}{(1+t^2)^{n/2}}\, d\th\nonumber\\
&\le& \tilde c_\e \intl_{\bbr} dt
\intl_{\snm1} \frac{(R\,\U_a |f|) (\th, t)}{(1+t^2)^{n/2}}\, d\th.\nonumber\eea
Hence, by (\ref{duas3}) and (\ref{fyqatrz}),
\[
I_\e \le  \tilde c_\e \intl_{\bbr^n}
\frac{(\U_a |f|)(x)}{(1+|x|^2)^{1/2}}\,dx \le \tilde c_\e\, c_a\intl_{\sn_a} |f (\eta)|\, d\eta.\]
Thus there is a constant $c(\e, a)=\tilde c_\e\, c_a>0$ such that
\be\label{bfgrv}
||\V_a R \,\U_a f||_{L^1(\bbs^{n, \e}_+)} \le  c(\e, a) \,||f||_{L^1(\sn_a)}.\ee

 To complete the proof, we observe that the estimates (\ref{bfgrv}) and (\ref{bfeggrv}) in conjunction with  (\ref{brv}) on the dense subset  $C_c (\sn_a) \subset L^1(\sn_a)$ imply (\ref{brv}) for all $f\in L^1(\sn_a)$.
\end{proof}

In the case $a=1$, that was excluded in the previous theorem,  it is easier to prove an analogue of  (\ref{brv})
directly, by making use of the stereographic projection.
\begin{theorem} \label {jaau2}
If
\be\label{brtrv}
\intl_{\sn}\frac{|f (\eta)|}{\sqrt {1-\eta_{n+1}}}\, d\eta <\infty,\ee
 then, for almost all $\xi \in \sn_+$,  $(\S_1 f)(\xi)$ is finite   and
\be\label{kiaw} (\S_1 f)(\xi) =   (\V_1 R\,\U_{1} f)(\xi).\ee
\end{theorem}
\begin{proof}  We first show that $(\V_1 R\,\U_{1} |f|)(\xi)<\infty$ for almost all $\xi \in \sn_+$. Then we prove  (\ref{kiaw}). This will give us the a.e. finiteness of $(\S_1 f)(\xi)$, and the proof will be complete.

By (\ref{duas3}) and (\ref{fhrz}),
\[ \intl_{\snm1 \times \bbr} \!\!\frac{(R\, \U_1f)(\theta,
t)}{(1\!+\!t^2)^{n/2}}\,d_*\theta dt\!=\! \intl_{\bbr^n}\!
\frac{(\U_1f)(x)}{(1\!+\!|x|^2)^{1/2}}\,dx\le c \intl_{\sn} \!\frac{|f (\eta)|}{\sqrt {1\!-\!\eta_{n+1}}}\, d\eta<\infty.\]
It follows that $(R\, \U_1f)(\theta,
t)<\infty$ for almost all $(\theta,
t)$, and therefore $(\V_1 R\,\U_{1} |f|)(\xi)<\infty$ for almost all $ \xi\in\sn_+$. The latter makes
 all subsequent operations with integrals  well-justified.

Let $\xi= \th  \sin\psi+ e_{n+1} \cos \psi $ $\th \in \snm1$, $0<\psi \le \pi/2$.
Because both $\S_1 $ and $R$ commute with rotations about the $x_{n+1}$-axis, it suffices to assume $\th=e_n=(0, \ldots, 0,1,0)$. Let $\t_\gam$ be the hyperplane containing the slice  $\gam=\gam(e_n, \psi)$, and let
 \be \label {sVES}
o'=e_n \cos \psi \sin \psi +e_{n+1}\, \cos^2 \psi \in \t_\gam\ee
be the center of the sphere $\gam$.
We translate $\t_\gam$ so that $o'$ moves to the origin $o=(0, \ldots, 0)$. Then we rotate the translated plane $\t_\gam -o'$, making it coincident with the coordinate plane $e_n^\perp$ and keeping the  subspace $\bbr^{n-1}=\bbr e_1 \oplus \ldots \oplus \bbr e_n$ fixed.  Let $\tilde\gam \subset e_n^\perp$ be the image of  $\gam$ under this transformation. We stretch $\tilde\gam$ up to the unit sphere $\sn$ in $e_n^\perp$ and
project $\sn$ stereographically onto  $\bbr^{n-1}$.
Thus, can write
\be \label {sVES1} \gam=o'+\rho\tilde\gam, \quad \rho=\left[ \begin{array} {cc} I_{n-1} &  0  \\ 0 & \rho_\psi
\end{array}\right], \quad \rho_\psi=\left[ \begin{array} {cc} \sin \psi &  -\cos \psi  \\ \cos \psi & \sin \psi
\end{array}\right], \ee
\[(\S_1 f) (e_n, \psi)\!=\!\intl_{\tilde\gam } f(o'\!+\!\rho\sig)\, d_{\tilde\gam }\sig\!=\!r^{n-1}\intl_{\sn}  f(o'\!+\!\rho r\sig)\, d\sig, \quad r\!=\!\sin \psi.\]
By (\ref{pq301}) (with $n$ replaced by $n-1$), we obtain
\[(\S_1 f) (\th, \psi)=(4r)^{n-1}\intl_{\bbr^{n-1}} f(o'+\rho r\tilde \pi (y))\, \frac{dy}{(|y|^2 +4)^{n-1}}, \]
 \[\tilde\pi (y)=\frac{4y+(|y|^2-4)\,e_{n+1}}{|y|^2+4}. \]

By  (\ref{sVES}) and (\ref{sVES1}),
\[o'\!+\!\rho r\tilde \nu (y)=\frac{4y \sin \psi+4e_n \sin 2 \psi + e_{n+1} (|y|^2 + 4\cos 2\psi)}{|y|^2+4}. \]
 Hence,
\bea
&&(\S_1 f) (e_n, \psi)=(4\sin \psi)^{n-1}\nonumber\\
&&\times \intl_{\bbr^{n-1}} \!\! \!f
\left (\frac{4y \sin \psi+4e_n \sin 2 \psi + e_{n+1} (|y|^2 \!+\! 4\cos 2\psi)}
{|y|^2+4}\right )\,\frac{dy}{(|y|^2 \!+\!4)^{n-1}}. \nonumber\eea
Changing variable $y=u  \sin \psi$ and setting $t=\cot \psi$, we represent this expression  as
\[4^{n-1}\!\!\intl_{\bbr^{n-1}} \!\!\! f\left (\frac{4(u\!+\!2te_n)\!+\!(|u\!+\!2te_n|^2 \!-\!4) \, e_{n+1}}{|u+2te_n|^2 +4}\right)\, \frac{du}{(|u\!+\!2te_n|^2 \!+\!4)^{n-1}}.\]
The latter is the Radon transform  $(R\,\U_1 f)(e_n, 2t)$; cf. (\ref{xhyan1}).
\end{proof}

Theorems \ref{jaau2} and  \ref{jaau1} enable us to investigate the slice transform $\S_a $ using  properties of the hyperplane Radon transform $R$. Some results of this kind are presented below. Others, like range characterization, support theorem, etc. are left to the reader.

Owing to (\ref{brv}) and (\ref{kiaw}), inversion of   $\S_a$   reduces to inversion of $R$ by the formula
\be\label{bthbc}
 \S_a^{-1}  =   \U_a^{-1} R^{-1}\,\V_{a}^{-1},\ee
where $\U_a^{-1}$ and $\V_{a}^{-1}$  have the form (\ref{ytrz}) and (\ref{ytrz1}), respectively.
Numerous  formulas for  $R^{-1}$ are available in the literature; see \cite {H11, GGG2, Ru15} and references therein. Applicability of these formulas essentially depends on the class of functions. If the action of $\S_a$ is considered on the entire space
$L^1(\sn_a)$ (for $0\le a<1$) or on functions $f$ satisfying (\ref{brtrv}) (for $a=1$), then $R$ must be inverted on the weighted $L^1$ space
\be\label{bfecv}
\tilde L^1(\rn)=\Big \{ \tilde f:\, \intl_{\rn} \frac{ |\tilde f(x)|}{(1+|x|^2)^{1/2}}\, dx <\infty\Big \}.\ee
Pointwise inversion formula for $R$ on this entire space need additional technical work. Luckily,  there is a plenty of inversion formulas when  $\tilde f \in L^p(\rn)$ with $1\le p<n/(n-1)$, in particular, for smooth $\tilde f$. Using H\"older's inequality, one can easily check that this $L^p$ space continuously embeds in $\tilde L^1(\rn)$.  If we restrict the action of $R$ to $\tilde f \in L^p(\rn)$, the class  of the original functions $f$ must be restricted accordingly to Corollaries \ref{dlfluig}, \ref{dlfluig1}. In particular, (\ref{brrv}) and (\ref{bcddrv1}) give the following statement.

\begin{theorem} \label {llhericeg1}

A function $f$ can be uniquely reconstructed from $\S_a f$ by the formulas (\ref{bthbc}),  (\ref{ytrz}), and (\ref{ytrz1}) under the following assumptions.

\vskip 0.2 truecm

\noindent In the case {\rm $-1< a<1$:}
\be\label{gYiov}  (a-\eta_{n+1})^{n-(n+1)/p} f(\eta) \in L^p (\sn_a), \qquad 1\le p< \frac{n}{n-1}.\ee

\noindent In the case {\rm $a=1${\rm :}}
\be\label{gYiov1}
(1-\eta_{n+1})^{n-1-n/p} f(\eta) \in L^p (\sn), \qquad 1\le p< \frac{n}{n-1}.\ee
In particular,

\vskip 0.2truecm
\noindent {\it for $a=0$}{\rm :}

\be\label{gbDDDCo1}
f(\eta)= \frac{1}{|\eta_{n+1}|^{n}} (R^{-1}\Psi_0)\left ( \frac{\eta'}{|\eta_{n+1}|}\right ),\ee
\[ \Psi_0(\th, t)=\frac{1}{\sqrt {t^2+1}} \,(\S_0 f)\left ( \frac{\theta +te_{n+1}}{\sqrt{t^2 +1}} \right ); \]

\noindent {\it for $a=1$}{\rm :}
\be\label{gbDDDCo2}
f(\eta)= \left (\frac{2}{1-\eta_{n+1}}\right )^{n-1} (R^{-1} \Psi_1)\left ( \frac{2\eta'}{1 -\eta_{n+1}}\right ),\ee
\[ \Psi_1(\th, t)=(\S_1 f)\left (\frac{2\th +te_{n+1}}{\sqrt{t^2 +4}}\right ). \]
\end{theorem}

Here the inversion formulas (\ref{gbDDDCo1}) and (\ref{gbDDDCo2}) follow from (\ref{ffnw21}), (\ref{ffan21}), (\ref{ffn21b}), and   (\ref{ffn21}).

The following statement gives an example of the inversion  formula for the  Radon transform; cf. \cite[Theorem 4.66]{Ru15}.
\begin{theorem} \label {l88g1}
A function $f \in L^p (\bbr^n)$ with $1\le p<n/(n-1)$ and $n$  odd, can be recovered from its Radon transform  by the  formula
\be\label{nnxxzz}f(x) =  \lim\limits_{t\to 0}\, \pi^{(1-n)/2} \left(-\frac {1}{2t}\,\frac {d}{dt}\right )^{(n-1)/2}\intl_{\snm1} (Rf)(\theta, t+x\cdot
\theta)\,d_*\theta,\ee
 where the  limit  is understood in the $L^p$-norm.
\end{theorem}

Many  inversion formulas that work for diverse classes of functions and  any parity of $n$ can be found in \cite[Section 4.8]{Ru15}.

\subsection{Slice Transforms of Zonal Functions}

The case when
 $f$ is a zonal function and $\S_a f$  can be expressed as a one-dimensional  Abel type integral, is worth being mentioned separately. Looking at the formula $\S_a f =   \V_a R\,\U_{a} f$, we observe that $\U_{a}$ takes zonal function on $\sn_a$ to radial functions on $\rn$, $R$ takes radial functions to single variable functions on $\bbr_+$, and the latter are transformed by $\V_a$ to zonal function on $\sn_+$. To make things precise,  we set
\be\label{asiurq} f(\eta)\equiv f_0 (\eta_{n+1}), \qquad (\U_{a} f)(x)\equiv (\stackrel{\circ}{\U}_{a} f_0)(|x|); \ee
\be\label{asiurqq} g(x)\equiv g_0 (|x|), \qquad    (Rg)(\th, t) \equiv (\stackrel{\circ}{R} g_0)(t); \ee
\be\label{asiurq43q}  \Phi (\th, t)\equiv  \Phi_0 (t), \quad t>0;  \qquad   (\V_a \Phi)(\xi)=(\stackrel{\circ}{\V}_{a} \Phi_0)(\xi_{n+1}). \ee
Here (cf. (\ref{hyan2b}), (\ref{hyanb}), (\ref{psq1tr})),
\be\label{asiurqq2}
(\stackrel{\circ}{\U}_{a} f_0)(r)=\frac{(a+1)\, P_0^{n-1} (r)}{D_0 (r)}\, (f_0 \circ \pi_0)(r),\ee
\[
D_0(r)=\sqrt{a+1}\,\sqrt{r^2 (1-a) +a+1}, \qquad P_0 (r)=\frac{a(a+1)+D_0(r)}{r^2 +(a+1)^2},\]
\[
\pi_0(r)\equiv Q_0(r)=\frac{ar^2 - (a+1)D_0(r)}{r^2 +(a+1)^2};\]
\be\label{asiurqq5}
(\stackrel{\circ}{R} g_0)(t)=\sigma_{n-2} \intl^\infty_{t}\! g_0 (r)
(r^2-t^2)^{(n-3)/2}\,r dr, \quad t>0;\ee
\be\label{asiurqq6}
(\stackrel{\circ}{\V}_{a} \Phi_0)(s)=
\sqrt{\frac{1\!-\!a^2s^2}{1\!-\!s^2}} \Phi_0 \left (\frac{(a+1)s}{\sqrt {1-s^2}}\right ), \quad s>0.\ee

Thus we have the following statement.
\begin{proposition} Let $f$ be a zonal function on $\sn_a$, $-1< a\le 1$. Suppose that $f(\eta)\equiv   f_0 (\eta_{n+1})$, so that
\be\label {ew9i} \intl_{-1}^a f_0 (t) (1-t^2)^\del\, dt<\infty,\ee
where $\del=(n-2)/2$ if $-1< a< 1$,  and  $\del=(n-3)/2$ if $a=1$. Then $(\S_a f)(\xi)=(\stackrel{\circ}{\S}_{a} f_0)(\xi_{n+1})$, where
\be\label {ew9i1}
(\stackrel{\circ}{\S}_{a} f_0)(s) =   (\stackrel{\circ}{\V}_{a}  \stackrel{\circ}{R}\,\stackrel{\circ}{\U}_{a} f_0)(s), \qquad 0< s<1.\ee
\end{proposition}
 The condition (\ref {ew9i}) is a reformulation for zonal functions of the assumption for $f$ in Theorems \ref{jaau1} and  \ref{jaau2}.

The  cases $a=0$ and $a=1$ in (\ref {ew9i1}) are of particular  interest.

\begin{corollary} Let $f(\eta)\equiv   f_0 (\eta_{n+1})$.\\

\noindent {\rm (i)} If $a=0$,  then  $(\S_0 f)(\xi)=F_0 (\sqrt {1-\xi_{n+1}^2})$, where
\be\label{ieuutw}
F_0 (\rho)=\frac{\sigma_{n-2}}{\rho^{n-2}} \intl_0^\rho f_0 (-t)(\rho^2 -t^2)^{(n-3)/2}\, dt.\ee

\noindent {\rm (ii)} If $a=1$, then  $(\S_1 f)(\xi)=F_1 (2\xi_{n+1}^2 -1)$, where
\be\label{ieu55utw}
F_1 (\rho)=\sigma_{n-2}\,\left (\frac{1\!- \!\rho}{2}\right )^{(3-n)/2}  \intl_\rho^1 f_0 (t) (t\!-\!\rho)^{(n-3)/2}\, (1\!-\!t)^{n-3}\,dt.\ee
\end{corollary}
\begin{proof}  {\rm (i)} Owing to (\ref{xhyan0}),  (\ref{asiurqq5}), and (\ref{xhyan20}), from (\ref{ew9i1}) we have
\bea
&&(\stackrel{\circ}{\S}_{0} f_0)(s) =\frac{\sigma_{n-2}}{\sqrt{1\!-\!s^2}} \intl^\infty_{t}\! g_0 (r)
(r^2-t^2)^{(n-3)/2}\,r dr\nonumber\\
&& \left (\text{here} \quad t\!=\!\frac{s}{\sqrt{1\!-\!s^2}}, \quad  g_0 (r)\!=\!\frac{1}{(r^2 \!+\!1)^{n/2}} f_0\left(-\frac{1}{\sqrt{r^2\! +\!1}} \right)\right )\nonumber\\
 &&=\frac{\sigma_{n-2}}{\sqrt{1\!-\!s^2}} \! \intl^\infty_{s/\sqrt{1\!-\!s^2}} \!\!\!\! f_0\left(-\frac{1}{\sqrt{r^2\! +\!1}} \right)
\left (r^2\!-\!\frac{s^2}{1\!-\!s^2} \right )^{(n-3)/2} \!\frac{rdr}{(r^2\! +\!1)^{n/2}}.\nonumber\eea
Changing variables, we obtain (\ref{ieuutw}).

 {\rm (ii)} We proceed as above, by making use of (\ref{xhyan1}), (\ref{xhyan21}), (\ref{asiurqq5}), and (\ref{ew9i1}):
\bea
&&(\stackrel{\circ}{\S}_{1} f_0)(s) =
\sigma_{n-2} \intl^\infty_{t}\! g_0 (r) (r^2-t^2)^{(n-3)/2}\,r dr\nonumber\\
&& \left (\text{here} \quad t\!=\!\frac{2s}{\sqrt{1\!-\!s^2}}, \quad  g_0 (r)\!=\!\left (\frac{4}{r^2 \!+\!4}\right )^{n-1}
f_0\left( \frac{r^2 \!-\!4}{r^2\! +\!4} \right )\right )\nonumber\\
 &&=\sigma_{n-2}\! \intl^\infty_{2s/\sqrt{1\!-\!s^2}} \!\!\!\! \left (\frac{4}{r^2 \!+\!4}\right )^{n-1}\!\!
f_0\left( \frac{r^2 \!-\!4}{r^2\! +\!4} \right )
\left (r^2\!-\!\frac{4s^2}{1\!-\!s^2} \right )^{(n-3)/2} \!rdr.\nonumber\eea
Changing variables, we obtain (\ref{ieu55utw}).
\end{proof}

\section{Appendix}

It is  known  that if
 $f$ is  integrable  on $[0,\infty)$ and
 continuous from the right at $t=0$, then
 \be\label{t224}
\lim\limits_{\a \to 0+} {1\over \Gamma (\a)} \intl^\infty_0 t^{\a-1}
f(t) \,dt =f(0);\ee
see, e.g., \cite[Theorem 2.41]{Ru15}.
We need the following generalization of this statement.

\begin{lemma} \label {weakenedz} Let $0<\a_0<1$, and let $u(\a,t)$ be a measurable function on $(0,\a_0] \times (0,\infty)$, satisfying the following conditions.
\vskip 0.2 truecm

\noindent {\rm (a)}$ \; \lim\limits_{(\a,t) \to (0,0)} u(\a,t)=u_0$.

${}$

\noindent {\rm (b)} There exists $c>0$ such that
\be \label{app0}\intl^\infty_0 |u(\a,t)|\, dt \le c \quad \text{for all}\quad \a\in  (0,\a_0].  \ee
Then
 \be\label{t224}
\lim\limits_{\a \to 0} {1\over \Gamma (\a)} \intl^\infty_0 t^{\a-1}
u(\a,t) \,dt =u_0.\ee
\end{lemma}
\begin{proof} Choose any $\e>0$ and let $\del\equiv \del (\e)<1$ be such
that
\be \label{app1}
|u(\a,s)- u_0|<\e \quad \text{\rm whenever} \quad 0<\a+s<\del. \ee
We set
\[ \vp_\a (t)=\intl_0^t u(\a,s)\, ds, \qquad \psi_\a(t)=\frac{1}{t}\intl_0^t [u(\a,s))-u_0)]\,
ds,\]
so that
\be \label{app2} \vp_\a (t)=t\,[\psi_\a(t)+ u_0]\ee
and, by (\ref{app1}),
\be \label{app3}  |\psi_\a(t)| <\e \quad \text{\rm whenever} \quad 0<\a+t<\del. \ee
 We have
\bea  {1\over \Gamma (\a)} \intl^\infty_0 t^{\a-1}
u(\a,t) \,dt &=&
{1\over \Gamma (\a)} \intl^\del_0
t^{\a-1} u(\a,t) dt + {1\over \Gamma (\a)} \intl^\infty_\del t^{\a-1}
u(\a,t) dt\nonumber\\
&=& I_1 (\a) +I_2(\a).\nonumber\eea
By (\ref{app0}),
\[
|I_2 (\a)| \le {1\over \Gamma (\a)} \intl^\infty_\del t^{\a-1}
|u(\a,t)| dt \le  {\del^{\a-1}\over \Gamma (\a)}  \intl^\infty_\del
|u(\a,t)| dt \le {c\, \del^{\a-1}\over \Gamma (\a)} \to 0\]
as $\a \to 0$.
To estimate $I_1 (\a)$, we proceed  as follows.
\bea I_1 (\a)&=&{1\over \Gamma (\a)} \intl^\del_0 t^{\a-1} d\vp_\a
(t)=\frac{t^{\a-1}\vp_\a(t)}{\Gamma (\a)}\Bigg |_0^\del
+\frac{1-\a}{\Gamma (\a)}\intl^\del_0
t^{\a-2}\vp_\a (t)\, dt\nonumber\\
&=&\frac{\del^{\a-1}\vp_\a (\del)}{\Gamma (\a)}-\lim\limits_{t \to 0}
\frac{t^{\a-1}\vp_\a(t)}{\Gamma (\a)}+
\frac{1-\a}{\Gamma (\a)}\intl^\del_0 t^{\a-1}\,[\psi_\a(t)+u_0]\, dt\nonumber\\
&=&\frac{\del^{\a-1}\vp_\a (\del)}{\Gamma (\a)}-\lim\limits_{t \to 0}
\frac{t^{\a}[\psi_\a(t)+u_0]}{\Gamma (\a)}+
\frac{1-\a}{\Gamma (\a)}\intl^\del_0
t^{\a-1}\,[\psi_\a(t)+u_0]\, dt.\nonumber\eea
By (\ref{app3}), the limit in this expression is zero, and therefore
\[ I_1 (\a)=\frac{\del^{\a-1}\vp_\a (\del)}{\Gamma (\a)}+
\frac{1-\a}{\Gamma
(\a)}\intl^\del_0
t^{\a-1}\,\psi_\a(t)\, dt +u_0\, \frac{\del^{\a} (1-\a)}{\Gamma (\a +1)}.\]
Hence
 \bea I_1 (\a) -u_0&=&
u_0\left [\frac{\del^{\a} (1-\a)}{\Gamma (\a+1)}-1\right
]+\frac{\del^{\a-1}\vp_\a (\del)}{\Gamma (\a)}+\frac{1-\a}{\Gamma
(\a)}\intl^\del_0 t^{\a-1} \psi_\a(t)\, dt\nonumber\\
&=&  A_1 (\a)+A_2 (\a)+ A_3 (\a).\nonumber\eea
Clearly, $  A_1 (\a)\to 0$, as $\a\to 0$. Further, by (\ref{app0}),
\[
|A_2 (\a)|=\frac{\del^{\a-1}}{\Gamma (\a)} \,\Bigg | \intl^\del_0 u(\a,s)\, ds \Bigg |\le \frac{c\,\del^{\a-1}}{\Gamma (\a)}\to 0\]
as $\a\to 0$.
It remains to note that by (\ref{app3}),
\[
|A_3 (\a)| \le \frac{\e\, (1-\a)\, \del^\a}{\Gamma
(\a+1)} \to \e, \]
as $\a\to 0$.
Finally, \[\lim\limits_{\a \to 0}|I_1 (\a) -f(0)|\le
\e.\] Since $\e>0$ is arbitrarily, the proof is complete.
\end{proof}

The next statement is a  consequence of Lemma \ref{weakenedz}.
For the classical Riesz potential
\be\label{rpot} (I^\a g)(x)=\frac{1}{\gamma_n(\a)}
\intl_{\bbr^n}
 \frac{g(y)\,dy}{|x-y|^{n-\a}},\qquad
  \gamma_n(\a)=
  \frac{2^\a\pi^{n/2}\Gamma(\a/2)}{\Gamma((n-\a)/2)},
  \ee
it is known that $ \lim\limits_{\a\to 0} \,(I^\a g)(x)=g (x)$ if $g$ is good enough; see, e.g.,   \cite [Lemma 3.2]{Ru15}.
  A generalization of this fact to more general integrals
\be\label{rpotv} V_\a (x)=\frac{1}{\gamma_n(\a)}
\intl_{\bbr^n}
 \frac{v(\a, y)}{|x-y|^{n-\a}}\,dy\ee
is the following.

\begin{lemma}\label {riesz-enLEM}  Let $0<\a_0<1$, $n \ge 1$, and let $v$ be a measurable function on $(0,\a_0] \times \bbr^n$. Given a point $x\in \rn$, suppose that there is a ball $B_x(r)=\{y\in\rn: \, |x-y|\le r \}$, $r>0$, and nonnegative functions  $c_1 (x)$ and $c_2 (x)$ such that

\vskip 0.2 truecm

\noindent {\rm (a)}$\;  |v(\a, y)| \le c_1 (x)$ for all $(\a, y) \in (0,\a_0] \times B_x(r)$,

\vskip 0.2 truecm

\noindent {\rm (b)}
\be \label{app0s} \intl_{\rn \setminus B_x(r)} \frac{|v(\a, y)|}{|x-y|^{n-2}}\, dy \le c_2 (x)\quad \text{for all} \quad \a\in  (0,\a_0].  \ee
If  $\lim\limits_{(\a, y)\to (0,x)} v(\a, y)=v (x)$, then $\lim\limits_{\a\to 0} \,V_\a (x)=v (x)$.
\end{lemma}
\begin{proof} Changing variables and passing to polar coordinates, we obtain
\bea V_\a(x)&=&\frac{\sig_{n-1}} {\gamma_n(\a)} \intl_0^\infty r^{\a-1}\, dr \intl_{\sn} v(\a, x-r\th)\, d_*\th\nonumber\\
&=&\frac {\sig_{n-1}\,\Gamma((n-\a)/2)} {2^{\a+1}\,\pi^{n/2}}\left [\frac{1}{\Gam (\a/2)}\intl_0^\infty t^{\a/2 -1} u_x(\a,t)\,dt\right ],\nonumber\eea
where
\be \label{app0st} u_x(\a,t)=\intl_{\snm1} v(\a, x-\sqrt {t}\, \th)\, d_*\th.\ee
Because the constant in front of the square bracket equals $1$ when $\a=0$, it remains to show that the expression inside the  brackets falls into the scope of Lemma \ref{weakenedz} with $u_0=v(x)$.
By the assumption (a), we can pass to the limit under the sign of integration in (\ref{app0st}) to get  $\lim\limits_{(\a,t) \to (0,0)} u_x(\a,t)=v(x)$. The validity of (\ref{app0}) is checked as follows.
\bea
&&\intl^\infty_0 |u_x(\a,t)|\, dt\le \frac{2}{\sig_{n-1}}\intl^\infty_0 r\, dr \intl_{\snm1} |v(\a, x-r \th)|\, d\th\nonumber\\
&&=\frac{2}{\sig_{n-1}}\,\Bigg ( \,\intl_{B_x(r)} +  \intl_{\rn \setminus B_x(r)}\Bigg )\,  \frac{|v(\a, y)|}{|x-y|^{n-2}}\, dy
\nonumber\\
&&\le \frac{2}{\sig_{n-1}}\,\Bigg ( c_1(x) \intl_{B_x(r)}\frac{dy}{|x-y|^{n-2}} +  \intl_{\rn \setminus B_x(r)}\frac{|v(\a, y)|}{|x-y|^{n-2}}\, dy\Bigg).\nonumber\eea
By the assumptions (a) and (b), the last expression is finite and the proof is complete.
\end{proof}

Below we show two applications of Lemma \ref{riesz-enLEM} that have been used in the present paper.

\begin{lemma}\label {WWWLEM} Let
\be \label{cosx} (W_\lam)(\xi) =\gam_{n,\lam}
 \intl_{\sn} w(\lam, \eta)\, |\xi \cdot \eta|^{\lam} \,d_*\eta, \ee
  \[\gam_{n,\lam}=\frac{\pi^{1/2}\,\Gamma( -\lam/2)}{\Gamma ((n+1)/2)\, \Gamma ((1+\lam)/2)},\quad -1<\lam \le 0, \quad n\ge 2,\]
where $w$ is a continuous function on $[-1,0] \times \sn$. If
  $\lim\limits_{\lam\to -1} w(\lam, \eta)=w_0(\eta)$,
 then $\lim\limits_{\lam\to -1} (W_\lam)(\xi) =c_n \,(\F w_0) (\xi)$, where $c_n= \pi^{1/2}/\Gamma (n/2)$ and  $\F$ is the Funk transform (\ref{fu}).
\end{lemma}
\begin{proof} By (\ref{clop10}) and (\ref{sphm0}),
\be\label{parq1}
W_\lam(\xi)=\frac{\gam_{n,\lam}\, \sig_{n-1}}{\sig_{n}} \intl_{-1}^1  |y|^\lam v_\xi (\lam, y)\, dy,\ee
where
\be\label{parq2} v_\xi (\lam, y)=(1-y^2)^{(n-2)/2} \intl_{\sn\cap \xi^\perp}
\!\!w(\lam,  \sqrt{1-y^2}\, \psi +y\xi)\,d_*\psi.\ee
We can now apply  Lemma \ref {riesz-enLEM}  with $\a=\lam +1$, $n=1$, $x=0$,  to the integral (\ref{parq1}), assuming that $v_\xi (\lam, y)$  extends by zero outside $[-1,1]$. By the continuity of $w$,
\[ \lim\limits_{(\lam, y) \to (-1,0)} v_\xi (\lam, y)=\intl_{\sn\cap \xi^\perp}
\!\!w_0(\psi)\,d_*\psi=(\F w_0) (\xi)\]
and $|v_\xi (\lam, y)|\le \const$ uniformly in $(\lam, y)\in [-1,0] \times [-1,1]$. The latter gives (a) in  Lemma \ref {riesz-enLEM}. The validity of (b) in that lemma is trivial because $v_\xi (\cdot, \cdot)$ is uniformly bounded  and compactly supported for every $\xi$. The constant $c_n$ can be easily calculated. Hence the result follows.
\end{proof}

Our next example is related to the  Semyanistyi type integral
\be \label{sosa}  \R_\a (\th, t) =\frac{1}{\gam_1 (\a)}
 \intl_{\rn} \rho(\a, x)\, |x \cdot \th -t|^{\a -1} \,dx,\ee
\[\gamma_{1}(\a)=
  \frac{2^\a\pi^{1/2}\Gamma(\a/2)}{\Gamma((1-\a)/2)}, \qquad 0<\a<1, \]
and  the relevant  Radon transform  (\ref{rad}).

\begin{lemma}\label {WWWny} Let $0<\a_0 <1$, $B_r =\{x\in \rn :\, |x|\le r\}$, $r>0$, and let  $\rho$ be a function on $[0,\a_0] \times \rn$ satisfying the following conditions.

\vskip 0.2 truecm

\noindent {\rm (a)}  For each $\a\in [0,\a_0] $, $\; \rho (\a, x)=0$ whenever $x \notin B_r$;

\vskip 0.2 truecm

\noindent {\rm (b)}   $\; \rho$ is continuous on $[0,\a_0] \times B_r$;

\vskip 0.2 truecm

\noindent {\rm (c)}   $\lim\limits_{\a\to 0} \, \rho (\a, x)=\rho_0 (x)$.

\vskip 0.2 truecm

\noindent Then $\lim\limits_{\a\to 0}  \R_\a (\th, t) =(R\rho_0)(\th, t)$.
\end{lemma}
\begin{proof} By Fubini's theorem,
\be \label{sguoty}
\R_\a (\th, t) =\frac{1}{\gam_1 (\a)} \intl_{\bbr} |y-t|^{\a -1}\, v_\th(\a,y)\,dy, \ee
where
\be \label{sgoty} v_\th(\a,y)=\intl_{x \cdot \th =y}\rho(\a, x)\, dx=\intl_{\th^\perp}  \rho(\a, y\th +u)\, d_\th u.\ee
Our aim is to show that the integral (\ref{sguoty}) falls into the scope of Lemma \ref{riesz-enLEM} with $n=1$.
If $|u|>r$, then $|y\th +u|=\sqrt{|u|^2 +|u|^2}>r$ and therefore, by (a),   $\rho(\a, y\th +u)=0$. Hence
\be \label{soty}
v_\th(\a,y)=\intl_{\th^\perp \cap B_r}  \rho(\a, y\th +u)\, d_\th u.\ee
This equality, in conjunction with (b) and (c), allows us to pass to the limit in (\ref {sgoty}) and obtain
\be \label{isgoty}
\lim\limits_{(\a,y)\to (0,t)} v_\th(\a,y)=\intl_{\th^\perp}  \rho_0(t\th +u)\, d_\th u=(R\rho_0)(\th, t).\ee
The latter gives the limit assumption in Lemma \ref{riesz-enLEM}. Further, by (\ref{soty}),
\be \label{iuuy}
|v_\th(\a,y)|\le \intl_{\th^\perp \cap B_r}  |\rho(\a, y\th +u)|\, d_\th u \le c\, v_{n-1} (r),\ee
where $c= \max\limits_{(\a, x)\in [0,\a_0] \times B_r} |\rho (\a, x)|$ and $v_{n-1} (r)$ is the volume of the $(n-1)$-dimensional ball of radius $r$. This implies (a) in  Lemma \ref{riesz-enLEM}, uniformly in $t$. Note also that  by (\ref{iuuy}),
\bea
\intl_{\bbr\setminus [t-r, t+r]}\! \!|v_\th(\a,y)| \,|y-t|\, dy&\le& r\intl_{\bbr\setminus [t-r, t+r]}\!\! dy\intl_{\th^\perp \cap B_r} \!\! |\rho(\a, y\th +u)|\, d_\th u\nonumber\\
&\le& 2cr^2 v_{n-1} (r). \nonumber\eea
Thus the integral in (\ref{sguoty}) satisfies all conditions of  Lemma \ref{riesz-enLEM}, and the result follows.
\end{proof}

\bibliographystyle{amsplain}

\end{document}